\documentclass{amsart}

\usepackage{amssymb}
\usepackage{verbatim} 

\newtheorem{theorem}{Theorem}[section]
\newtheorem{lemma}[theorem]{Lemma}
\newtheorem{proposition}[theorem]{Proposition}
\newtheorem{corollary}[theorem]{Corollary}

\theoremstyle{definition}

\newtheorem{example}[theorem]{Example}

\theoremstyle{remark}

\numberwithin{equation}{section}

\newcommand{\bfind}[1]{{\bf #1}}
\newcommand{\n}{\par\noindent}

\newcommand{\sn}{\par\smallskip\noindent}

\newcommand{\pars}{\par\smallskip}
\newcommand{\parm}{\par\medskip}

\newcommand{\ovl}[1]{\overline{#1}}

\newcommand{\sep}{^{\rm sep}}

\newcommand{\chara}{\mbox{\rm char}}

\newcommand{\appr}{\mbox{\rm appr}}
\newcommand{\dist}{\mbox{\rm dist}}
\newcommand{\supp}{\mbox{\rm supp}}

\newcommand{\Gal}{\mbox{\rm Gal}\,}

\newcommand{\gloss}[1]{\glossary{#1}{\sl #1}}
\newcommand{\restr}{|^{ }_}
\newcommand{\subsetuneq}{\mathrel{\raisebox{.8ex}{\footnotesize%
$\displaystyle\mathop{\subset}_{\not=}$}}}
\newcommand{\cal}{\mathcal}
\newcommand{\N}{\mathbb{N}}
\newcommand{\Q}{\mathbb{Q}}

\newcommand{\F}{\mathbb{F}}
\newcommand{\bh}{{\bf h}}
\newcommand{\bhsc}{\mbox{\scriptsize\bf h}}

\begin{document}

\title[Relative approximation degree]{The relative approximation degree
in valued function fields}

\author{Franz-Viktor Kuhlmann and Izabela Vlahu}
\address{Department of Mathematics and Statistics,
   University of Saskatchewan,
   106 Wiggins Road,
   Saskatoon, Saskatchewan, Canada S7N 5E6}
\email{fvk@math.usask.ca, izabelavlahu@gmail.com}
\thanks{The first author wishes to thank Peter Roquette for his
invaluable help and support during the preparation of his doctoral
thesis, in which several of the results presented here appeared first.
He also wishes to thank F.~Delon, B.~Green, A.~Prestel and F.~Pop for
inspiring discussions and suggestions. Thanks also to Salih Durhan for
his careful reading of the manuscript.
\\ During the work on this paper, the first author  was
partially supported by a Canadian NSERC grant.}


\date{1.\ 8.\ 2012}

\begin{abstract}
We continue the work of Kaplansky on immediate valued field extensions
and determine special properties of elements in such extensions. In
particular, we are interested in the question when an immediate valued
function field of transcendence degree 1 is henselian rational (i.e.,
generated, modulo henselization, by one element). If so, then wild
ramification can be eliminated in this valued function field. The
results presented in this paper are crucial for the first author's proof
of henselian rationality over tame fields, which in turn is used in his
work on local uniformization.
\end{abstract}

\maketitle

%
%
\section{Introduction}\label{s:introduction}
This paper continues the work of Kaplansky \cite{KAP1} in which, based
on earlier work of Ostrowski \cite{OS3}, he laid the foundations for an
understanding of immediate extensions of valued fields. Such an
understanding has turned out to be essential for many questions about
the structure of valued fields, which vary from their model theory or
applications in real algebra to the very difficult task of elimination
of wild ramification in valued function fields. The latter plays an
essential role in the quest for local uniformization, which in turn is a
local form of resolution of singularities. These problems being still
wide open in positive characteristic, any refined valuation
theoretical tools that can bring new insight are very important.

The theory developed by Kaplansky and Ostrowski is very useful for
valuations with residue fields of characteristic 0, but its real
strength (as well as its limitations) become visible when the residue
characteristic is positive.

While Kaplansky was mainly concerned with embeddings in power series
fields and the question when maximal immediate extensions are unique up
to isomorphism, the above mentioned problems have added new questions to
the spectrum. In the present paper we develop Kaplansky's tools further
in order to answer various questions about the structure of immediate
function fields. Several results of this paper are indispensable for the
paper \cite{HR} on henselian rationality, which is central in the first
author's work on elimination of wild ramification and local
uniformization (see \cite{KK2}), as well as the model theory of valued
fields (see \cite{FVK}).

By $(L|K,v)$ we denote an extension of valued fields, i.e., $L|K$ is a
field extension, $v$ is a valuation on $L$, and $K$ is endowed with the
restriction of $v$ (which will again be denoted by $v$.) An extension
$(L|K,v)$ is said to be {\bf immediate} if the canonical embeddings
$vK\hookrightarrow vL$ of the value groups and $Kv\hookrightarrow Lv$ of
the residue fields are onto. An important example for an immediate
algebraic extension of a valued field $(K,v)$ is its henselization,
denoted by $(K,v)^h$ or just $K^h$, which is a minimal extension in
which Hensel's Lemma holds. An immediate function field $(F|K,v)$ of
transcendence degree 1 will be called {\bf henselian rational}
if there exists an element $x\in F^h$ such that $F^h=K(x)^h$,
that is, $F^h$ is the henselization of the rational function field
$K(x)$. (This eliminates wild ramification from $(F|K,v)$.) We then call
$x$ a \bfind{henselian generator} of $F^h$.

The main theorem of \cite{HR} states that every immediate function field
$(F|K,v)$ of transcendence degree 1 over a tame field $(K,v)$ is
henselian rational. The field $(K,v)$ is called {\bf tame} if it is
henselian (i.e., $K=K^h$) and the ramification group of the extension
$K^{\sep}|K$, where $K^{\sep}$ is the separable algebraic closure of
$K$, is trivial, that is, the fixed field $K^r$ of this ramification
group is algebraically closed.

For the proof of this theorem, one first reduces the problem to
the case of valued fields of rank 1 (i.e., having archimedean
ordered value groups), and then starts with an arbitrary element
$x\in F$ transcendental over $K$; it can be chosen such that $F|K(x)$
is separable. If $x$ is not a henselian generator, then $(F^h|K(x)^h,v)$
is a proper finite immediate extension. Let us describe the further
steps of the proof in the important special case where $\chara K=p>0$.
If one replaces $(F|K,v)$ by the valued function field $(F.K^r| K^r,v)$,
which again is immediate, then the extension $((F.K^r)^h| K^r(x)^h,v)$
becomes a tower of Artin-Schreier extensions. The lowest of them is
shown to be generated by a root $y$ of a polynomial $X^p-X-f(x)$ where
$p$ is the residue characteristic and $f(x)\in K[x]$. We observe that
$f(x)= y^p-y\in K(y)$, hence if $K(x)^h= K(f(x))^h$, then $K(x)^h
\subsetuneq K(y)^h$. Replacing $x$ by $y$, we have then reduced the
degree of $F^h|K(x)^h$ by a factor of $p$. This shows that it is crucial
to determine the degree $[K(x)^h : K(f(x))^h]$ for a given $f(x)\in
K[x]$ and to choose $f(x)$ in such a way that the degree becomes 1.

In order to gain insight on the degree $[K(x)^h : K(f(x))^h]$, we study
the elements $f(x)\in K[x]$ in (not necessarily transcendental)
immediate extensions $(K(x)|K,v)$, through extending Kaplansky's
technical lemmas. After introducing approximation types and their basic
properties in Sections~\ref{sectatdef} and~\ref{sectiat}, this study is
carried out in in Sections~\ref{sectpiat} to~\ref{sectad}. In
Section~\ref{sectrad}, we define the ``relative approximation degree of
$f(x)$ in $x$'' to be the integer $h$ that appears in Kaplansky's Lemma
8. We then show in Theorem~\ref{,7} that under suitable assumptions
about the extension $(K(x)|K,v)$ and the element $f(x)$, the degree
$[K(x)^h : K(f(x))^h]$ is smaller than or equal to the relative
approximation degree of $f(x)$ in $x$.

Having proved (in \cite{HR}) that the immediate function field
$(F.K^r|K^r(x),v)$ is henselian rational, one has to pull this property
down to $(F|K,v)$. Observe that if $(F.K^r|K^r(x),v)$ is henselian
rational, then the same already holds for $(F.L|L(x),v)$ with is a
suitable finite subextension $L|K$ of $K^r|K$. Moreover, $L|K$ can be
chosen to be Galois since also $K^r|K$ is Galois (we allow Galois
extensions to be infinite). An extension of a henselian field $(K,v)$ is
called \bfind{tame} if it lies in $K^r$. Consequently, a Galois
extension is tame if and only if its ramification group is trivial. So
what we need is a pull down principle for henselian rationality through
tame extensions of the base field. This is presented in
Theorem~\ref{,pdp}. More precisely, we show in Section~\ref{+8.2} that
if $x$ is a henselian generator for $(F.L|L,v)$, where $(L|K,v)$ is a
finite tame Galois extension, then for a suitable element $d\in L$, the
trace $\mbox{\rm Tr} (d\cdot x)$ is a henselian generator for $(F|K,v)$.
We use a valuation theoretical characterization of the Galois groups of
tame Galois extensions that is developed in Section~\ref{sectvigp}.

Once a henselian generator $x\in F^h$ is found, the question
arises whether $x$ can already be chosen in $F$. We show in
Theorem~\ref{hr-down} that this can be done. In fact, there is some
$\gamma\in vK$ such that $K(x)^h=K(y)^h$ for every $y\in F$ with
$v(x-y)\geq\gamma$. This result is crucial for
the proof given in \cite{KK2} that local uniformization can always be
achieved after a finite Galois extension of the function field. In order
to prove Theorem~\ref{hr-down}, we generalize the relative approximation
degree to other elements $y\in K(x)^h$ in place of $f(x)$ in
Section~\ref{,ad2}. We then prove the corresponding generalization
of Theorem~\ref{,7}: Theorem~\ref{,6} states that under suitable
assumptions, we again have that the degree $[K(x)^h:K(y)^h]$ is smaller
than or equal to the relative approximation degree of $y$ in $x$.

Theorem~\ref{hr-down} can be seen as a special case of a
``dehenselization'' procedure (analogous to the ``decompletion'' used by
M.~Temkin in \cite{Tem}). If for a given valued function field $(F|K,v)$
there is a finite extension $F'$ of $F$ within its henselization such
that $(F'|K,v)$ admits local uniformization, one would like to deduce
that also $(F|K,v)$ admits local uniformization. This can be done if
Theorem~\ref{hr-down} can be generalized in a suitable way to the case
of non-immediate valued function fields. This problem will be
investigated in a subsequent paper.

\pars
Our investigation of the properties of elements in immediate extensions
is facilitated by the introduction of the notion of ``approximation type'',
which we use in place of Kaplansky's ``pseudo-convergent sequences''
(also called ``pseudo-Cauchy sequences'' or ``Ostrowski nets'' in the
literature). This new notion makes computations and the formulation of
results easier. For instance, to every element $x$ in an immediate
extension $(L|K,v)$, we associate the unique approximation type of $x$
over $K$, while there are many pseudo-convergent sequences in $K$ that
have $x$ as a pseudo-limit, and in addition one needs to require
maximality of such sequences (for $x\notin K$ one asks that they do not
have a pseudo limit in $K$). Furthermore, the definition of
approximation types is not restricted to immediate extensions only. In
fact, approximation types can be further enhanced to a tool for
describing properties of elements in non-immediate extensions. In
Section~\ref{sectKth}, we take the occasion to show how Kaplansky's
fundamental Theorems 2 and 3 can be proved by using approximation types
in place of pseudo-convergent sequences.


This paper is based on results that appeared in the first author's
doctoral thesis (cf.~\cite{FVKth}) and presents updated, improved and
extended versions of them, with simplified proofs. The preparation of
these results for publication is part of the second author's Masters
thesis.

%
%
%
\section{Some preliminaries}                \label{sectprel}
%
%
For basic facts from valuation theory, see \cite{End}, \cite{[Eng--P]},
\cite{RIB1}, \cite{W}, \cite{ZS}.

Take a valued field $(K,v)$. We denote its value group by $vK$, its
residue field by $Kv$, and its valuation ring by ${\cal O}_K$. For $a\in
K$, we write $va$ for its value and $av$ for its residue.

By $\tilde{K}$ we will denote the algebraic closure of $K$. For each
extension of $v$ to $\tilde{K}$, we have that $\tilde{K}v=
\widetilde{Kv}$, and $v\tilde{K}$ is the divisible hull of $vK$, which
we denote by $\widetilde{vK}$.

Note that the extension $(L|K,v)$ is immediate if and only if for all
$b\in L$ there is $c\in K$ such that $v(b-c)>vb$ (as is implicitly shown
in the proof of Lemma~\ref{imme} below). This property can be used to
define immediate extensions of other valued structures, such as valued
abelian groups and valued vector spaces.

An algebraic extension $(L|K,v)$ of henselian fields is called
\bfind{defectless} if every finite subextension $E|K$ satisfies the
fundamental equality $[E:K]={\rm e} \cdot {\rm f}$, where e$\,=(vE:vK)$
is the ramification index and f$\,= [Ev:Kv]$ is the inertia degree. In
this case, $(E|K,v)$ admits a \bfind{standard valuation basis}, which we
construct as follows: we take $a_1,\ldots,a_{\rm e}\in E$ such that
$va_1+vK, \ldots, va_{\rm e}+vK$ are the cosets of $vK$ in $vE$, and
$b_1,\ldots,b_{\rm f}\in E$ such that $b_1v,\ldots,b_{\rm f}v$ are a
basis of $Ev|Kv$. Then $a_ib_j\,$, $1\leq i\leq\,$e, $1\leq j\leq\,$f,
is a basis of $E|K$, and it has the following property: for all choices
of $c_{ij}\in K$,
\[
v\sum_{i,j}c_{ij}a_ib_j\>=\>\min_{i,j} vc_{ij}a_ib_j
\>=\>\min_{i,j} vc_{ij}a_i \;.
\]
Note that we can always choose $a_1=b_1=1$ so that $a_1b_1=1$.

All tame extensions of henselian fields are defectless, see \cite{FVK}.
The following facts are well known and easy to prove:

\begin{lemma}                               \label{apprdle}
Take a defectless extension $(L|K,v)$ of henselian fields and $a\in L$.
Then the set $\{v(a-c)\mid c\in K\}$ has a maximum. More precisely, if
we choose a standard valuation basis for $E=K(a)$ as above and write
\[
a\>=\>\sum_{i,j}c_{ij}a_ib_j\;,
\]
then $v(a-c_{1,1})$ is the maximum of $\{v(a-c)\mid c\in K\}$.
\end{lemma}

We will also need the following tool (cf.\ \cite[Lemma 2.5]{Kasd}):

\begin{lemma}                               \label{idllind}
Take a henselian field $(K,v)$, a valued field extension $(K'|K,v)$, an
immediate subextension $(F|K,v)$, and a defectless algebraic
subextension $(L|K,v)$. Then $F|K$ and $L|K$ are linearly disjoint,
$(F.L|F,v)$ is defectless, and $(F.L|L,v)$ is immediate.
\end{lemma}

%
%
\section{Approximation types and distances}  \label{sectatdef}
We will now introduce approximation types, which constitute a suitable
structure for dealing with immediate extensions of valued fields.

We define $B_\alpha (c,K) = \{a\in K \mid v(a-c)\geq\alpha\}$ to be the
``closed'' ultrametric ball in $(K,v)$ of radius $\alpha\in vK\infty:=vK
\cup\{\infty\}$ centered at $c\in K$. An \bfind{approximation type}
{\bf over} $(K,v)$ is a full nest of closed balls in $(K,v)$, that is,
a collection
\[
{\bf A}\>=\> \{B_\alpha (c_\alpha , K) \mid \alpha\in S\}
\]
with $S$ an initial segment of $vK\infty$, $c_\alpha\in K$, and the
balls $B_\alpha (c_\alpha , K)$ linearly ordered by inclusion. We write
\glossary{${\bf A}_\alpha$} ${\bf A}_\alpha=B_\alpha (c_\alpha , K)$ for
$\alpha\in S$, and ${\bf A}_\alpha=\emptyset$ otherwise. We call $S$ the
\bfind{support} of ${\bf A}$ and denote it by $\supp{\bf A}$.

Note that if $\beta<\alpha\in\supp{\bf A}$, then ${\bf A}_\beta=B_\beta
(c_\beta,K)= B_\beta (c_\alpha,K)$, i.e., ${\bf A}_\beta$ is uniquely
determined. Hence, {\bf A} is uniquely determined by the balls
${\bf A}_\alpha$ where $\alpha$ runs through an arbitrary cofinal
sequence in $\supp{\bf A}$.

Take any extension $(L|K,v)$ and $x\in L$. For all
$\alpha\in vK\infty$, we set
\begin{equation}                            \label{atalpha}
\appr(x,K)_{\alpha} \;:=\; \{c\in K\mid v(x-c)\geq\alpha\}
\>=\>B_\alpha(x,L)\cap K\;.
\end{equation}
It is easy to check that $\appr(x,K)_{\alpha}$ is empty or a closed ball of
radius $\alpha$. If $\appr(x,K)_{\alpha}\ne\emptyset$ and
$\beta<\alpha$, then also $\appr(x,K)_{\beta}\ne\emptyset$. This shows
that the set
\[
\{\alpha \in vK\infty\mid\appr(x,K)_{\alpha}\ne\emptyset\}
\]
is an initial segment of $vK\infty$ and therefore,
\begin{equation}
\appr(x,K)\;:=\;\{\appr(x,K)_{\alpha}\mid \alpha\in vK\infty
\mbox{ and } \appr(x,K)_{\alpha}\ne\emptyset\}
\end{equation}
is an approximation type over $(K,v)$. We call $\appr(x,K)$ the
\bfind{approximation type of $x$ over $(K,v)$}.

As the support $S$ of $\appr(x,K)$ is an initial segment of $vK\infty$,
$S\cap vK=S\setminus\{\infty\}$ is an initial segment of $vK$ and thus
induces a cut in $vK$ with lower cut set $S\setminus\{\infty\}$. Now
this cut induces a cut in the divisible hull $\widetilde{vK}$ of $vK$,
where the lower cut set is the smallest initial segment of
$\widetilde{vK}$ containing $S\setminus\{\infty\}$. We call this
cut the \bfind{distance of $x$ from $(K,v)$} and denote it by
\[
                        \dist(x,K)\;.
\]
We write $\dist(x,K)=\infty$ if the lower cut set is $\widetilde{vK}$,
and $\dist(x,K)< \infty$ otherwise. Note that $\dist(x,K)=\infty$ if and
only if $S$ contains $vK$, which holds if and only if $x$ lies in the
completion of $(K,v)$.

For a subset $A\subset K$ we define $\dist_K(x,A)$, the distance of
$x$ from $A$ over $K$, to be the cut in $\widetilde{vK}$ having as lower
cut set the smallest initial segment in $\widetilde{vK}$ containing the
set $\{v(x-c)\mid c\in A\}\cap vK$.

Note that if $(L|K,v)$ is an algebraic extension of valued fields, then
the divisible hull of $vK$ coincides with the divisible hull of $vL$ and
so for an element $x$ in an extension of $K$, we have that $\dist(x,K)$
and $\dist(x,L)$ are both cuts in the same group. This allows us to
compare these distances by set inclusion of the lower cut sets. Another
reason to take the distance in the divisible hull is that the
classification of Artin--Schreier defect extensions through
distances presented in \cite{Kasd} does not work if they are taken in
ordered abelian groups with archimedean components which are not dense;
this situation does not appear in divisible groups.

If $n$ is a natural number and the lower cut set of $\dist(x,K)$ is $D$,
then
\[
n\cdot\dist(x,K)
\]
will denote the cut with lower cut set $nD:=\{n\gamma\mid\gamma\in D\}$;
note that $nD$ is again an initial segment of $\widetilde{vK}$ because
of divisibility.

If $C$ and $C'$ are two cuts in a linearly ordered set $T$ defined by
their lower cut sets $D$ and $D'$, respectively, then $C=C'$ if $D=D'$,
and we write $C<C'$ if $D\subsetuneq D'$. For an element $\alpha\in T$
we write $\alpha>C$ if $\alpha>\beta$ for all $\beta\in D$, and
$\alpha\geq C$ if $\alpha\geq \beta$ for all $\beta\in D$; note that if
$D$ has no last element, then $\alpha>C\Leftrightarrow \alpha\geq C$. We
write $\alpha\leq C$ if $\alpha\in D$, and $\alpha<C$ if $\alpha\in D$
but is not the last element of $D$.

\begin{lemma}                                       \label{v>Lambda}
Take an extension $(L|K,v)$ of valued fields, and $x,x'\in L$.
\sn
a) \ For every $\alpha$ in the support of $\appr(x,K)$,
$\appr(x,K)_\alpha=\appr(x',K)_\alpha$ holds if and only if
$v(x-x')\geq\alpha$.\sn
b) \ Further,
\begin{eqnarray}
&&\appr(x,K)=\appr(x',K) \>\Longrightarrow\> v(x-x')\geq\dist(x,K)
=\dist(x',K)\>,                              \label{v>Leq1} \\
&&v(x-x')\geq\max\{\dist(x,K),\dist(x',K)\}
\>\Longrightarrow\> \appr(x,K)=\appr(x',K)\>. \label{v>Leq2}
\end{eqnarray}
\end{lemma}
\begin{proof}
a):\ \ Take $\alpha\in vK\infty$. If $v(x-x')\geq\alpha$, then $B_\alpha
(x,L) =B_\alpha (x',L)$, which yields that $\appr(x,K)_\alpha =B_\alpha
(x,L) \cap K=B_\alpha (x',L)\cap K =\appr(x',K)_\alpha\,$. If
$v(x-x')<\alpha$, then $B_\alpha (x,L)\cap B_\alpha (x',L)=\emptyset$,
whence $\appr(x,K) _\alpha  \cap  \appr(x',K)_\alpha=\emptyset$; for
$\appr(x,K)_\alpha \not = \emptyset$, this yields that $\appr(x,K)_
\alpha\not=\appr(x',K)_\alpha$.
\sn
b):\ \ If $\dist (x,K)\ne\dist (x',K)$, then $\appr(x,K)\ne
\appr(x',K)$. If $v(x-x')\geq\dist (x,K)$ does not hold, then there is
some $\alpha$ in the support of $\appr(x,K)$ such that $\alpha>v(x-x')$.
By part a), it follows that $\appr(x,K)_\alpha\not=\appr(x',K)_\alpha$.
This proves (\ref{v>Leq1}).

If $v(x-x')\geq\dist(x,K)$ holds, then $v(x-x')\geq\alpha$ for all
$\alpha\ne\infty$ in the support of $\appr(x,K)$. Again by part a), it
follows that $\appr(x,K)_\alpha=\appr(x',K)_\alpha$ for all $\alpha\ne
\infty$ in the support of $\appr(x,K)$. Similarly, $v(x-x')\geq\dist(x',
K)$ implies that $\appr(x,K)_\alpha=\appr(x',K)_\alpha$ for all $\alpha
\ne \infty$ in the support of $\appr(x',K)$. If none of the supports
contains $\infty$, then we obtain that $\appr(x,K)=\appr(x',K)$. If on
the other hand, at least one support contains $\infty$, then the
corresponding distance is $\infty$, whence $v(x-x')=\infty$, i.e.,
$x=x'$ and again, $\appr(x,K)=\appr(x',K)$. We have proved
(\ref{v>Leq2}).
\end{proof}

If {\bf A} is an approximation type over $(K,v)$ and there exists an
element $x$ in some valued extension field $L$ such that ${\bf A} =
\appr(x,K)$, then we say that \bfind{$x$ realizes {\bf A}} (in $(L,v)$).
If {\bf A} is realized by some $c\in K$, then {\bf A} will be called
{\bf trivial}.\index{trivial approximation type} This holds if and only
if ${\bf A}_\infty\ne \emptyset$, in which case ${\bf A}_\infty=\{c\}$.

We leave the easy proof of the following lemma to the reader.

\begin{lemma}                               \label{x:realat}
Take an approximation type {\bf A} over $(K,v)$ and an extension
$(L|K,v)$ of valued fields. The element $x\in L$ realizes {\bf A} if and
only if the following conditions hold:
\sn
\ 1) \  if $\alpha\in\supp{\bf A}$, then $v(x-c)\geq\alpha$ for some
$c\in {\bf A}_\alpha$,
\n
\ 2) \ if $\beta\notin\supp{\bf A}$, then $v(x-c)<\beta$ for all
$c\in K$.
\end{lemma}
%

\pars
For our work with approximation types, we introduce the
following notation which is particularly useful in the immediate case.
We introduce it in connection with valued fields, but its application to
ultrametric spaces and other valued structures is similar. So take
an arbitrary valued field $(K,v)$ and an approximation type {\bf A}
over $(K,v)$. Further, take a formula $\varphi$ with one free
variable. Then the sentence
\[
       \varphi(c) \mbox{ for } c\nearrow {\bf A}
\]
will denote the assertion
\[
\mbox{there is }\alpha\in vK \mbox{ such that }{\bf A}_{\alpha}\ne
\emptyset \mbox{ and } \varphi(c) \mbox{ holds for all }
c\in{\bf A}_{\alpha}\;.
\]
Note that if $\varphi_1(c)$ for $c\nearrow {\bf A}$ and $\varphi_2(c)$
for $c\nearrow {\bf A}$, then also $\varphi_1(c)\wedge \varphi_2(c)$ for
$c\nearrow {\bf A}$.

In the case of ${\bf A}=\appr(x,K)$, we will also write ``$c\nearrow
x$'' in place of ``$c\nearrow {\bf A}$''.

If $\gamma=\gamma(c)\in vK$ is a value that depends on $c\in K$ (e.g.,
the value $vf(c)$ for a polynomial $f\in K[X]$), then we will say
that \bfind{$\gamma$ increases for $c\nearrow x$} if there exists some
$\alpha \ne \infty$ in the support of $\appr(x,K)$ such that for every
choice of $c'\in \appr(x,K)_{\alpha}$ with $x\not=c'\,$,
\[
    \gamma(c)>\gamma(c') \mbox{\ \ for\ \ } c\nearrow x\;.
\]
Note that the condition $x\not=c'$ is automatically satisfied if
$\appr(x,K)$ is nontrivial.

%
%
\section{Immediate approximation types}     \label{sectiat}
An approximation type {\bf A} with support $S$ will be called
{\bf immediate} \index{immediate approximation type} if its intersection
\[
\bigcap {\bf A}\>=\>\bigcap_{\alpha\in S}{\bf A}_\alpha
\]
is empty. If {\bf A} is trivial, then $\bigcap {\bf A}={\bf A}_\infty\ne
\emptyset$; therefore, an immediate approximation type is never trivial.

\begin{lemma}                               \label{imme}
Let $(L|K,v)$ be an extension of valued fields.
\sn
a) If $x\in L$, then $\appr(x,K)$ is immediate if and only if for every
$c\in K$ there is some $c'\in K$ such that $v(x-c')>v(x-c)$, that is,
the set
\[
v(x-K) \>:=\> \{v(x-c) \mid c \in K\}
\]
has no maximal element.
\sn
b) The extension $(L|K,v)$ is immediate if and only if for every $x\in
L\setminus K$, its approximation type $\appr(x,K)$ over $(K,v)$ is
immediate.
\sn
c)
%
If $\appr(x,K)$ is immediate, then its support is equal to $v(x-K)$.
\end{lemma}
\begin{proof}
a): \ Suppose that $\appr(x,K)$ is immediate and that $c$ is an
arbitrary element of $K$. Then by definition there is some
$\alpha$ such that $c\notin \appr(x,K)_\alpha\ne\emptyset$, so
$v(x-c)<\alpha$. Choosing some $c'\in \appr(x,K)_\alpha\,$, we obtain
that $v(x-c)< \alpha\leq v(x-c')$.

Now take $x\in L\setminus K$ and suppose that for every $c\in K$ there
is $c'\in K$ such that $v(x-c')>v(x-c)$. Then there is also some $c''\in
K$ such that $v(x-c'')>v(x-c')$. By the ultrametric triangle law we
obtain that $v(c'-c)=v(x-c)< v(x-c')= v(c''-c')$. Hence $v(c'-c)\in
v(x-K)$ and $c\notin \appr(x,K)_{v(c''-c')}\ne\emptyset$. As $c\in K$
was arbitrary, this shows that $\appr(x,K)$ is immediate.

\sn
b): \ Assume that $(L|K,v)$ is immediate. Take $x\in L\setminus K$ and
an arbitrary $c\in K$. Then $v(x-c)\in vL=vK$, i.e., there is $d\in K$
such that $v(x-c)=vd$ so that $vd^{-1}(x-c)=0$. Then $d^{-1}(x-c)v\in
Lv=Kv$, i.e., there is $d'\in K$ such that $d^{-1}(x-c)v=d'v$, which
means that $v(d^{-1}(x-c)-d')>0$. This implies that $v(x-c-dd')>vd=
v(x-c)$. Setting $c'=c+dd'$, we obtain $v(x-c')>v(x-c)$. By part a) it
now follows that $\appr(x,K)$ is immediate.

For the converse, assume that for every $x\in L\setminus K$,
$\appr(x,K)$ is immediate. By the proof of a), for every $c\in K$ we
have that $v(x-c)\in vK$, so in particular, $v(x-0)\in vK$; this shows
that $vL|vK$ is trivial. It remains to show that $Lv|Kv$ is trivial.
Take any $x\in L\setminus K$ with $vx=0$. Since $\appr(x,K)$ is
immediate, there is $c'\in K$ such that $v(x-c') > v(x-0)=vx$. From this
we obtain that $xv=c'v \in Kv$. Hence $Lv|Kv$ is trivial.

\sn
c): \ If $\alpha\in vK$ is an element of the support of $\appr(x,K)$,
then $\appr(x,K)_\alpha\ne\emptyset$, and so by (\ref{atalpha}), there
is $c\in K$ such that $v(x-c)\geq\alpha$. In the case of $v(x-c)=
\alpha$, we immediately see that $\alpha\in v(x-K)$. In the case of
$v(x-c)>\alpha$, choose some $d\in K$ with $vd=\alpha$; then
$v(x-(c+d))=vd=\alpha$, which again shows that $\alpha\in v(x-K)$.

For the converse inclusion, take $c\in K$. By the proof of part a),
there is $c'\in K$ such that $v(x-c)=v(c'-c)$, which shows that
$v(x-c)\in vK$. It follows from (\ref{atalpha}) that $c\in
\appr(x,K)_{v(x-c)}\,$, so $v(x-c)$ is in the support of $\appr(x,K)$.
\end{proof}

For immediate approximation types, we can improve part b) of
Lemma~\ref{v>Lambda}, and Lemma~\ref{x:realat}.

\begin{lemma}                                \label{v>L-imm}
Take an extension $(L|K,v)$ of valued fields, and
$x,x'\in L$. If $\appr(x,K)$ is immediate, then
\begin{equation}                            \label{v>L-immeq}
\appr(x,K)=\appr(x',K) \>\Longleftrightarrow\> v(x-x')\geq\dist(x,K)\;.
\end{equation}
\end{lemma}
\begin{proof}
We only have to prove the implication ``$\Leftarrow$''. As in the proof
of (\ref{v>Leq2}), we deduce from $v(x-x')\geq\dist(x,K)$ that $v(x-x')
\geq\alpha$ and $\appr(x,K)_\alpha=\appr(x',K)_\alpha$
for all $\alpha\ne\infty$ in the support of $\appr(x,K)$.
Since $\appr(x,K)$ is immediate, we also know that $\infty$ is not in
its support. It remains to show that $\appr(x',K)_\alpha=\emptyset$ for
every $\alpha$ not in the support of $\appr(x,K)$. If this were not
true, there would be $c\in K$ such that $v(x'-c)>\supp\appr(x,K)$. Since
also $v(x-x')>\supp\appr(x,K)$, we would obtain that $v(x-c)>\supp
\appr(x,K)$. But then $c\in\bigcap\appr(x,K)$, contradicting the
assumption that $\appr(x,K)$ is immediate.
\end{proof}

\begin{lemma}                               \label{x:realiat}
Take an immediate approximation type {\bf A} over $(K,v)$ and an
extension $(L|K,v)$. The element $x\in L$ realizes {\bf A} if and only
if for every $\alpha\in\supp{\bf A}$, $v(x-c)\geq \alpha$ for some $c\in
{\bf A}_\alpha$.
\end{lemma}
\begin{proof}
We have to show that for every immediate approximation type {\bf A},
condition 2) of Lemma~\ref{x:realat} holds if condition 1) holds. Assume
that $\beta \notin\supp{\bf A}$. Since the support is an initial segment
of $vK\infty$, this means that $\beta>\supp{\bf A}$. Take any $c\in K$.
Since {\bf A} is immediate, there is some $\alpha\in\supp{\bf A}$ such
that $c\notin {\bf A}_\alpha\,$. By condition 1), there is some $c'\in
{\bf A}_\alpha$ such that $v(x-c')\geq\alpha$. Now $v(x-c)\geq\alpha$
would imply that $v(c-c')\geq\min\{v(x-c),v(x-c')\}\geq\alpha$, whence
$c\in {\bf A}_\alpha\,$, a contradiction. It follows that $v(x-c)<\alpha
<\beta$. Hence condition 2) holds.
\end{proof}

\begin{corollary}                           \label{vx-cnf}
Take an immediate approximation type {\bf A} over $(K,v)$, an extension
$(L|K,v)$ of valued fields, and $x\in L$. If $v(x-c)$ is not fixed for
$c\nearrow{\bf A}$, then ${\bf A}=\appr(x,K)$.
\end{corollary}
\begin{proof}
Our assumption means that for all $\alpha\in\supp{\bf A}$ there are
$c,c'\in {\bf A}_\alpha$ such that $v(x-c')>v(x-c)$. This implies that
$v(x-c')>\min\{v(x-c),v(c-c')\}$, whence $v(x-c)=v(c-c')\geq\alpha$. Now
our assertion follows from the previous lemma.
\end{proof}

\parm
In the remainder of this section, we wish to explore how immediate
approximation types behave under valued field extensions $(L|K,v)$. Take
$x$ in some extension of $L$ such that $x\notin L$ and $\appr(x,K)$ is
immediate. Obviously,
\[
\dist(x,L)\>\geq\>\dist(x,K)
\]
and
\begin{equation}                            \label{atupeq}
\appr(x,K)_\alpha=B_\alpha(c_\alpha,K)\>\Longrightarrow\>
\appr(x,L)_\alpha= B_\alpha(c_\alpha,L)\;.
\end{equation}
If $\dist(x,L)=\dist(x,K)$, then by (\ref{atupeq}), $\appr(x,K)$ fully
determines $\appr(x,L)$. But if $\dist(x,L)>\dist(x,K)$, then
$\appr(x,K)$ does not provide enough information for those
$\appr(x,L)_\beta$ with $\beta>\dist(x,L)$.

\begin{lemma}                               \label{atup}
If in the above situation $(L|K,v)$ is a defectless extension, then
$\dist(x,L)=\dist(x,K)$ and by (\ref{atupeq}), $\appr(x,K)$ fully
determines $\appr(x,L)$.
\end{lemma}
\begin{proof}
Suppose that $\dist(x,L)>\dist(x,K)$. Then there is some $a\in L$ such
that $v(x-a)>\dist(x,K)$, which by (\ref{v>Leq2}) implies that
$\appr(a,K)=\appr(x,K)$, which is immediate. But by Lemma~\ref{apprdle},
$\{v(a-c)\mid c\in K\}$ has a maximum. This contradicts part a) of
Lemma~\ref{imme}.
\end{proof}

%
%
\section{Polynomials and immediate approximation types}\label{sectpiat}
Take an arbitrary polynomial $f\in K[X]$ and an approximation type
{\bf A} over $(K,v)$. We will say that \bfind{{\bf A} fixes the value of
$f$} if there is some $\alpha\in vK$ such that $vf(c)=\alpha$ for
$c\nearrow {\bf A}$. We will call an immediate approximation type
{\bf A} a \bfind{transcendental approximation type} if {\bf A}
fixes the value of every polynomial $f(X)\in K[X]$. Otherwise, {\bf A}
is called an \bfind{algebraic approximation type}. If
there exists any polynomial $f\in K[X]$ whose value is not fixed by
{\bf A}, then there exists also a monic polynomial of the same degree
having the same property (since this property is not lost by
multiplication with nonzero constants from $K$). If $f(X)$ is a monic
polynomial of minimal degree $\mbox{\bf d}$ such that {\bf A} does not
fix the value of $f$, then it will be called an \bfind{associated
minimal polynomial} for {\bf A}, and {\bf A} is said to be of {\bf
degree} \gloss{\bf d}.\index{degree of an approximation type}
We define the degree of a transcendental approximation type to be
$\mbox{\bf d}= \infty$. According to this terminology, an approximation
type over $K$ of degree $\mbox{\bf d}$ fixes the value of every
polynomial $f\in K[X]$ with $\deg f< \mbox{\bf d}$.
Note that an associated minimal polynomial $f$ for {\bf A} is always
irreducible over $K$. Indeed, if the degree of $g,h\in K[X]$ is smaller
than $\deg f$, then {\bf A} fixes the value of $g$ and $h$ and thus also
of $g\cdot h$. Since every polynomial $g\in K[X]$ of degree $\mbox{\bf
d}$ whose value is not fixed by {\bf A} is just a multiple $cf$ of an
associated minimal polynomial $f$ for {\bf A} (with $c\in K^\times$),
the irreducibility holds for every such polynomial as well.

We note that an immediate approximation type {\bf A} fixes the value of
every linear polynomial in $K[X]$. Indeed, for every $c\in K$ there is
$\alpha\in\supp {\bf A}$ such that $c\notin {\bf A}_\alpha$. Hence for
all $c',c''\in {\bf A}_\alpha$, $v(c'-c'')>v(c-c')$ and thus $v(c'-c)=
v(c''-c)$. This shows that {\bf A} fixes the value of $X-c$. We conclude
that the degree of an algebraic approximation type is not less than
$2$.

\parm
We will now study the behaviour of polynomials with respect to
immediate approximation types $\appr(x,K)$. We need the following lemma
for ordered abelian groups, which is a reformulation of Lemma~4 of
Kaplansky \cite{KAP1}. For archimedean ordered groups, it was proved
by Ostrowski \cite{OS3}.
\begin{lemma}                               \label{OST}
Take elements $\alpha_1,\ldots,\alpha_m$ of an ordered abelian group
$\Gamma$ and a subset $\Upsilon\subset\Gamma$ without maximal element.
Let $t_1,\ldots,t_m$ be distinct integers. Then there exists an element
$\beta\in \Upsilon$ and a permutation $\sigma$ of the indices
$1,\ldots,m$ such that for all $\gamma\in \Upsilon$, $\gamma\geq\beta$,
\[
\alpha_{\sigma(1)} + t_{\sigma(1)}\gamma >
\alpha_{\sigma(2)} + t_{\sigma(2)}\gamma > \ldots >
\alpha_{\sigma(m)} + t_{\sigma(m)}\gamma\;.
\]
\end{lemma}

For an arbitrary polynomial $f(X)=c_n X^n+c_{n-1}X^{n-1}+\ldots+c_0\,$,
we call

\begin{equation}                            \label{iderivative}
f_i(X)\>:=\> \sum_{j=i}^{n}\binom{j}{i} c_j X^{j-i}
\>=\> \sum_{j=0}^{n-i}\binom{j+i}{i} c_{j+i} X^{j}
\end{equation}
the \bfind{$i$-th formal derivative of $f$} and

\begin{eqnarray}
f(X) & = & \sum_{i=0}^{n} f_i(c) (X-c)^i \label{Taylorexp}\\
f_i(X) & = & \sum_{j=i}^{n} {j\choose i} f_j(c) (X-c)^{j-i}
\label{Taylorexpi}
\end{eqnarray}
the \bfind{Taylor expansions} of $f$ and $f_i$ at $c$.

\pars
If the immediate approximation type {\bf A} is of degree {\bf d} and
$f\in K[X]$ is of degree at most {\bf d}, then {\bf A} fixes the value
of every formal derivative $f_i$ of $f$ ($1\leq i\leq\deg f$), since
every such derivative has degree less than {\bf d}. So we can define
$\beta_i$ to be the fixed value $vf_i(c)$ for $c\nearrow
x$.\glossary{$\beta_i$} In certain cases, a derivative may be
identically 0. In this case, we have $\beta_i =\infty$. However, the
Taylor expansion of $f$ shows that not all derivatives vanish
identically, and the vanishing ones will not play a role in our
computations.

By use of Lemma~\ref{OST}, we can now prove:

\begin{lemma}                        \label{C4+}
Take an immediate approximation type ${\bf A}=\appr(x,K)$ of degree
{\bf d} over $(K,v)$ and $f\in K[X]$ a polynomial of degree at most
{\bf d}. Further, let $\beta_i$ denote the fixed value $vf_i(c)$
for $c\nearrow x$. Then there is a positive integer $\bh\leq\deg f$ such
that
\begin{equation}                                   \label{xf-}
\beta_{\bf h} + {\bf h}\cdot v(x-c)<\beta_i+i\cdot v(x-c)
\end{equation}
whenever $i\not={\bf h}$, $1\leq i\leq\deg f$ and $c\nearrow x$. Hence,
\begin{equation}                                \label{bhmin+}
v(f(x)-f(c)) = \beta_{\bhsc} + {\bf h}\cdot v(x-c)\;\;\;\mbox{ for }
c\nearrow x\;.
\end{equation}
Consequently, if {\bf A} fixes the value of $f$, then
\[
v(f(x)-f(c))>vf(x)=vf(c)\;\;\;\mbox{ for }c\nearrow x\;,
\]
and if {\bf A} does not fix the value of $f$, then
\[
vf(x)>vf(c)= \beta_{\bhsc} + {\bf h}\cdot v(x-c)\;\;\;\mbox{ for }
c\nearrow x\;.
\]
\end{lemma}
\begin{proof}
Set $n=\deg f$. We consider the Taylor expansion
\begin{equation}                  \label{Tv}
f(x) - f(c) = f_1(c)(x-c)+\ldots+f_n(c)(x-c)^n
\end{equation}
with $c\in K$. We have that $vf_i(c)(x-c)^i=\beta_i+i\cdot v(x-c)$ for
$c\nearrow x$. So we apply the foregoing lemma with $\alpha_i=\beta_i$
and $t_i=i$, and with $\Upsilon$ equal to the support of {\bf A} (which
has no maximal element since {\bf A} is an immediate approximation
type). We find that there is an integer $\bh\leq\deg f$ such that
$\beta_{\bf h}+ {\bf h} v(x-c) <\beta_i+iv(x-c)$ for $c\nearrow x$ and
$i\not={\bf h}$. This is equation (\ref{xf-}), which in turn
implies equation (\ref{bhmin+}).

If {\bf A} fixes the value of $f$, then $vf(x)\not=vf(c)$ is impossible
for $c\nearrow x$ since otherwise, the left hand side of (\ref{bhmin+})
would be equal to $\min\{vf(x),vf(c)\}$ and thus fixed while the right
hand side of (\ref{bhmin+}) increases for $c\nearrow x$. This
proves that $vf(x)=vf(c)$ and thus also $v(f(x)-f(c))\geq vf(x)$ for
$c\nearrow x$. But since the right hand side increases, we find that
$v(f(x)-f(c))>vf(x)$ for $c\nearrow x$.

If {\bf A} does not fix the value of $f$, then $vf(x)\not=vf(c)$ and
thus $v(f(x)-f(c))=\min\{vf(x),vf(c)\}$ for $c\nearrow x$. Since
$v(f(x)-f(c))$ increases for $c\nearrow x$ and $vf(x)$ is a constant,
the minimum must be $vf(c)$, and $vf(x)=vf(c)$ is impossible.
\end{proof}

If $g\in K[X]$ has a degree smaller than the degree of {\bf A}, then by
the foregoing lemma, the value of $g(x)$ in $(K(x),v)$ is given by
$vg(x)=vg(c)$ for $c\nearrow x$. Since $g(c)\in K$, that means that the
value of $g(x)$ is uniquely determined by {\bf A} and the restriction of
$v$ to $K$. If $g$ is a nonzero polynomial, then $g(c)\not=0$ for
$c\nearrow x$ (since there is a nonempty ${\bf A}_\alpha$ which does
not contain the finitely many zeros of $g$, as {\bf A} is immediate).
Consequently, $g(x)\not=0$, which shows that the elements
$1,x,\ldots,x^{{\bf d} -1}$ are $K$-linearly independent.

We even know that $v(g(x)-g(c))> vg(x)$ for $c\nearrow x$. This means
that $(K,v)\subset (K+Kx+\ldots+ Kx^{{\bf d}-1},v)$ is an immediate
extension of valued vector spaces. If ${\bf d}=[K(x):K]<\infty$, then
$K(x)=K[x]=K+Kx+\ldots+ Kx^{{\bf d}-1}$, and so the valued
field extension $(K(x)|K,v)$ is immediate. If ${\bf d}=\infty$, then
$(K,v) \subset (K[x],v)$ is immediate. But then again it follows that
the valued field extension $(K(x)|K,v)$ is immediate. Indeed, if
$v(g(x)-g(c)) >vg(x)$ and $v(h(x)-h(c))>vh(x)$, then $vg(x)=vg(c)$,
$vh(x)=vh(c)$ and
\begin{eqnarray*}
v\left(\frac{g(x)}{h(x)}-\frac{g(c)}{h(c)}\right) & = &
v\left[g(x)h(c)-g(c)h(x)\right]-vh(x)h(c)\\
 & = & v\left[g(x)h(c)-g(c)h(c)+g(c)h(c)-g(c)h(x)\right]-vh(x)h(c)\\
 & = & v\left[(g(x)-g(c))h(c)+g(c)(h(c)-h(x))\right]-vh(x)h(c)\\
 & > & vg(x)h(x) -vh(x)h(x) \;=\; v\,\frac{g(x)}{h(x)}\;.
\end{eqnarray*}
We have proved:

\begin{lemma}                                           \label{CK1}
Take an immediate approximation type ${\bf A}=\appr(x,K)$ of
degree {\bf d} over $(K,v)$. Then the valuation on the valued
$(K,v)$-vector subspace $(K+Kx+\ldots+ Kx^{\mbox{\scriptsize\bf d}
-1},v)$ of $(K(x),v)$ is uniquely determined by {\bf A} because
\[
vg(x)=vg(c)\;\;\;\mbox{ for }c\nearrow x
\]
for every $g(x)\in K+Kx+\ldots+ Kx^{{\bf d} -1}$. The elements
$1,x,\ldots,x^{{\bf d} -1}$ are $K$-linearly independent. In
particular, $x$ is transcendental over $K$ if ${\bf d}=\infty$.

Moreover, the extension $(K,v)\subset (K+Kx+\ldots+Kx^{{\bf d}-1},v)$
of valued vector spaces is immediate. In particular, if $\mbox{\bf d}=
\infty$ or if $\mbox{\bf d} =[K(x):K]<\infty$, then $(K[x]|K,v)$ is
immediate and the same is consequently true for the valued field
extension $(K(x)|K,v)$.
\end{lemma}

\pars
So far we have only considered polynomials of degree at most {\bf d};
the next lemma will cover the remaining case.

\begin{lemma}                               \label{deg>d} \label{f>d}
Take an immediate algebraic approximation type ${\bf A}=\appr(x,K)$ over
$(K,v)$ and an associated minimal polynomial $f\in K[X]$ for {\bf A}.
Further, take an arbitrary polynomial $g\in K[X]$ and write
\begin{equation}                            \label{repg}
g(X)=c_k(X)f(X)^k+\ldots+c_1(X)f(X)+c_0(X)
\end{equation}
with polynomials $c_i\in K[X]$ of degree less than $\deg f$. Then there
is some integer $m$, $1\leq m<k$, and a value $\beta\in vK$ such that
with $\bh$ as in Lemma~\ref{C4+},
\begin{equation}                                
v(g(c)-c_0(c)) = vc_m(c)+m\cdot vf(c)=\beta + m\cdot \bh\cdot v(x-c)
\;\;\;\mbox{ for } c\nearrow x\;.
\end{equation}
Consequently, if {\bf A} fixes the value of $g$, then
\[
vg(x) = vg(c) = vc_0(c) = vc_0(x) <
v(g(c)-c_0(c)) \;\;\;\mbox{ for }c\nearrow x\;,
\]
and if {\bf A} does not fix the value of $g$, then
\[
vg(x)>vg(c)=\beta+m\cdot\bh\cdot v(x-c)\;\;\;\mbox{ for } c\nearrow x\;.
\]
\end{lemma}
\begin{proof}
Since $\deg c_i(X)<\deg f(X)=\deg {\bf A}$, we have that {\bf A} fixes
the value of $c_i(X)$, for $0\leq i\leq k$. We denote by $\gamma_i$ the
fixed value $vc_i(c)$ for $c\nearrow x$. Since $f$ is an associated
minimal polynomial for {\bf A}, we know that {\bf A} does not fix the
value of $f$. From Lemma~\ref{C4+} we infer that the value of
$c_i(c) f(c)^i$ is equal to $\gamma_i+i\beta_{\bf h}+i{\bf h} v(x-c)$.
We apply Lemma~\ref{OST} with $\alpha_i=\gamma_i+i \beta_{\bf h}$,
$t_i=i{\bf h}$ and $\Upsilon=\supp{\bf A}$ to deduce that there is
an integer $m$ such that $0\leq m<k$ and $vc_m(c) f(c)^m <
vc_i(c)f(c)^i$ for $c\nearrow x$ and $1<i\not=m$. Consequently,
\begin{equation}                            \label{deg>deq}
v(g(c)-c_0(c))= vc_m(c)f(c)^m=\gamma_m+m\cdot\beta_{\bf h}+ m\cdot
{\bf h}\cdot v(x-c)\;.
\end{equation}
We set $\beta:=\gamma_m+m\beta_{\bf h}\,$.

The value of the right hand side of (\ref{deg>deq}) is not fixed for
$c\nearrow x$. Consequently, if {\bf A} fixes the value of $g$, then
from our representation (\ref{repg}) of $g$ we see that the value
$vc_m(c)f(c)^m$ must be greater than the fixed value of $c_0(c)$
for $c\nearrow x$, which yields that $vg(c) = vc_0(c)$. From
Lemma~\ref{C4+}, we know that $vc_0(x)=vc_0(c)$ and $vf(x)>vf(c)$
for $c\nearrow x$. Therefore,
\begin{equation}                            \label{vvvv}
vc_i(x)f(x)^i > vc_i(c)f(c)^i > vc_0(c)=vc_0(x)
\end{equation}
for $1\leq i\leq k$ and $c\nearrow x$, whence $vg(x)=vc_0(x)=vc_0(c)=
vg(c)$.

\pars
If {\bf A}\ does not fix the value of $g$, then
$vc_m(c)f(c)^m<vc_0(c)$ and
\[
vg(c) = vc_m(c)f(c)^m = \beta + m\cdot {\bf h}\cdot v(x-c)
\]
for $c\nearrow x$. The inequality $vg(x)>vg(c)$ for $c\nearrow x$,
is seen as follows. Using the first inequality of (\ref{vvvv}) together
with $vc_m(c)f(c)^m<vc_0(c)$, we obtain:
\begin{eqnarray*}
v g(x) & \geq &
\min\{v(c_k(x)f(x)^k)\,,\ldots,\,v(c_1(x)f(x))\,,\,vc_0(x)\}\\
& > & \min\{v(c_k(c)f(c)^k)\,,\ldots,\,v(c_1(c)f(c))\,,\,vc_0(c)\}
=vg(c)\;.
\end{eqnarray*}
This completes the proof of our lemma.
%
%
\end{proof}


\begin{corollary}                      \label{nst}
Take an immediate approximation type $\appr(x,K)$ over $(K,v)$. If $x$
is algebraic over $K$ with minimal polynomial $g\in K[X]$, then
$\appr(x,K)$ does not fix the value of $g$ and is thus of degree
$\mbox{\bf d} \leq [K(x):K]$.
\end{corollary}
\begin{proof}
Since $\appr(x,K)$ is immediate, it is nontrivial, so $x\notin K$ and
$g(c)\not=0$ for all $c\in K$. But by hypothesis, $g(x)=0$. Hence $vg(x)
> vg(c)$ for all $c\in K$. Now the assertion follows by an application
of Lemma~\ref{deg>d}.
\end{proof}

\n
Unfortunately, {\bf d} may be smaller than $[K(x):K]$, as the following
example will show:

\begin{example}                             \label{d<degree}
We choose $(K,v)$ to be $(\F_p(t),v_t)$ or $(\F_p((t)),v_t)$ or any
henselian intermediate field (where $\F_p$ is the field with $p$
elements). We take $L$ to be the perfect hull $K(t^{1/p^i}\mid i\in\N)$
of $K$.

If $\vartheta$ is a root of the polynomial
\[
X^p-X-\frac{1}{t}
\]
then the Artin--Schreier extension $L(\vartheta)|L$ is immediate with
$v(\vartheta-L)=\{\alpha\in vL\mid\alpha<0\}$ (see
\cite[Example 3.12]{Kdef}). It follows from Proposition~\ref{degat} below
and the fact that $(L,v)$ is henselian (being an algebraic extension of
the henselian field $(K,v)$) that $\mbox{\rm deg}\,\appr(\vartheta,L)
=p=[L(\vartheta):L]$. But an element $x=\vartheta+y$ in some
extension of $(L,v)$ has the same approximation type as $\vartheta$ over
$L$ if $vy\geq 0$ (cf.\ Lemma~\ref{v>Lambda}). We may take $y$ of
arbitrarily high degree over $L$. Indeed, we may even take $y$ to be
transcendental over $L$ to obtain that $\vartheta+y$ is transcendental
over $L$. This shows that a transcendental element may have an algebraic
approximation type. Moreover, we may choose $y$ such that $vy \notin vL$
or $yv \notin Lv$ to obtain an extension which is not immediate,
although its generating element has an immediate approximation type.
\end{example}

%
%
\section{Realization of immediate approximation types} \label{sectKth}
%
In this section we will present the two basic theorems due to
Kaplansky (\cite{KAP1}) which show that each immediate
approximation type can be realized in a simple immediate extension.
Kaplansky proved these theorems to derive a characterization of maximal
fields, which we will also present here.


\begin{theorem} {\rm\ \ \ (Theorem 2 of \cite{KAP1}, approximation
type version)}\label{KT2at}\n
For every immediate transcendental approximation type {\bf A} over
$(K,v)$ there exists a simple immediate transcendental extension
$(K(x),v)$ such that $\appr(x,K)={\bf A}$.

If $(K(y),v)$ is another valued extension field of $(K,v)$ such that
$\appr(y,K) = {\bf A}$, then $y$ is also transcendental over
$K$ and the isomorphism between $K(x)$ and $K(y)$ over $K$ sending
$x$ to $y$ is valuation preserving.
\end{theorem}
\begin{proof}
We take $K(x)|K$ to be a transcendental extension and define the
valuation on $K(x)$ as follows. In view of the rule $v(g/h)=vg-vh$, it
suffices to define $v$ on $K[x]$. Take $g\in K[X]$. By assumption,
{\bf A} fixes the value of $g$, that is, there is $\beta\in vK$ such
that $vg(c)=\beta$ for $c\nearrow {\bf A}$. We set $vg(x)=\beta$. If $g$
is a constant in $K$, we just obtain the value given by the valuation
$v$ on $K$. Our definition implies that $vg\not=\infty$ for every
nonzero $g\in K[x]$.

Take $g,h\in K[X]$. Again by our definition, $vg(x)=vg(c)$, $vh(x)=
vh(c)$, and $vg(x)h(x) = v(g\cdot h)(x)=v(g\cdot h)(c)=vg(c)h(c)$ for
$c\nearrow {\bf A}$. Thus, $vg(x)h(x)=vg(c)h(c)=vg(c) +vh(c)= vg(x)+
vh(x)$ and $v(g(x)+h(x))=v((g+h)(x))= v((g+h)(c))=
v(g(c)+h(c))\geq\min\{vg(c),vh(c)\}=
\min\{vg(x),vh(x)\}$ for $c\nearrow {\bf A}$. So indeed, our definition
yields a valuation $v$ on $K(x)$ which extends the valuation $v$ of $K$.
Under this valuation, we have that ${\bf A}=\appr(x,K)$. This is seen as
follows. In view of Lemma~\ref{x:realiat}, it suffices to prove that for
every $\alpha\in\supp{\bf A}$, we have that $v(x -c_\alpha) \geq\alpha$
for each $c_\alpha\in {\bf A}_\alpha$. But this follows directly from
our definition of $v(x-c_\alpha)$ because for $c\nearrow {\bf A}$, $c\in
{\bf A}_\alpha$ and thus $v(x-c_\alpha)=v(c-c_\alpha)\geq\alpha$.

From Lemma~\ref{CK1}, we now infer that $(K(x)|K,v)$ is an
immediate extension. Given another element $y$ in some valued field
extension of $(K,v)$ such that ${\bf A}=\appr(y,K)$, we want to show
that the epimorphism from $K[x]$ onto $K[y]$ induced by $x\mapsto y$ is
valuation preserving. For this, we only have to show that $vg(x)=vg(y)$
for every $g\in K[X]$. By hypothesis, the degree of {\bf A} is $\infty$.
From Lemma~\ref{CK1} we can thus infer that $vg(x)=vg(c)=vg(y)$ holds
for $c\nearrow {\bf A}$; this proves the desired equality. Again from
Lemma~\ref{CK1}, we deduce that $y$ is transcendental over $K$. Hence,
the assignment $x\mapsto y$ induces an isomorphism from $K(x)$ onto
$K(y)$. Since the valuations of $K(x)$ and $K(y)$ are uniquely
determined by its restriction to $K[x]$ and $K[y]$ respectively, it
follows from what we have already proved that this isomorphism is
valuation preserving.
\end{proof}

\begin{corollary}                           \label{corKT2}
Take an extension $(L|K,v)$ of valued fields and $y\in L$. If
$\appr(y,K)$ is an immediate transcendental approximation type, then $y$
is transcendental over $K$ and $(K(y)|K,v)$ is immediate.
\end{corollary}
\begin{proof}
By the foregoing theorem, there is an immediate extension $(K(x)|K,v)$
such that $\appr(x,K)=\appr(y,K)$, with $x$ transcendental over $K$. By
the same theorem, there is a valuation preserving isomorphism of $K(x)$
and $K(y)$ over $K$. This proves our assertions.
\end{proof}

The next lemma will show that every immediate algebraic approximation
type is of the form $\appr(y,K)$.

\begin{lemma}                               \label{nfvfreal}
Take an immediate algebraic approximation type {\bf A} over $(K,v)$,
a polynomial $f\in K[X]$ whose value is not fixed by {\bf A}, and a root
$y$ of $f$. Then there is an extension of $v$ from $K$ to $K(y)$ such
that ${\bf A}=\appr(y,K)$.
\end{lemma}
\begin{proof}
We choose some extension $w$ of $v$ from $K$ to $K(y)$.
We write $f(X)=d\prod_{i=1}^{\deg f}(X-a_i)$ with $d\in K$ and $a_i\in
\tilde{K}$. If for all $i$, the values $w(c-a_i)$ would be fixed for
$c\nearrow {\bf A}$, then {\bf A} would fix the value of $f$, contrary
to our assumption. Hence there is a root $a$ of $f$ such that $w(a-c)$
is not fixed for $c\nearrow{\bf A}$. Take some automorphism $\sigma$ of
$\tilde{K}|K$ such that $\sigma y=a$ and set $v:=w\circ\sigma$. Then $v$
extends the valuation of $K$, and $v(y-c)=w\circ\sigma(y-c)=w(\sigma
y-c)=w(a-c)$ is not fixed for $c\nearrow{\bf A}$. By
Corollary~\ref{vx-cnf}, ${\bf A}=\appr(y,K)$.
\end{proof}

The following is the analogue of Theorem~\ref{KT2at} for immediate
algebraic approximation types.

\begin{theorem} {\rm\ \ \ (Theorem 3 of \cite{KAP1}, approximation type
version)}\label{KT3at}\n
For every immediate algebraic approximation type {\bf A} over $(K,v)$ of
degree {\bf d} with associated minimal polynomial $f(X)\in K[X]$
and $\,y\,$ a root of $\,f$, there exists an extension of $v$ from
$K$ to $K(y)$ such that $(K(y)|K,v)$ is an immediate extension
and $\appr(y,K)={\bf A}$.

If $(K(z),v)$ is another valued extension field of $(K,v)$ such that
$\appr(z,K)={\bf A}$, then any field isomorphism between $K(y)$
and $K(z)$ over $K$ sending $y$ to $z$ will preserve the valuation.
(Note that there exists such an isomorphism if and only if $z$ is also a
root of~$\,f$.)
\end{theorem}
\begin{proof}
We take the valuation $v$ of $K(y)$ given by Lemma~\ref{nfvfreal}. Then
$\appr(y,K)={\bf A}$. The fact that $(K(y)|K,v)$ is immediate follows
from Lemma~\ref{CK1}.

%
%
%

The last assertion of our theorem is shown in the same way as the
corresponding assertion of Theorem~\ref{KT2at}: if $\appr(y,K)=\appr
(z,K)$ and $g\in K[X]$ with $\deg g < \mbox{\bf d}$ then, again by
Lemma~\ref{CK1}, $vg(y)=vg(c)= vg(z)$ for $c\nearrow x$. Hence an
isomorphism over $K$ sending $y$ to $z$ will preserve the valuation.
\end{proof}

From this theorem, we can derive important information about the degree
of immediate algebraic approximation types.

\begin{proposition}                         \label{degat}
The degree of an immediate algebraic approximation type over a
henselian field $(K,v)$ is a power of the characteristic of the residue
field $Kv$.
\end{proposition}
\begin{proof}
Take an immediate algebraic approximation type {\bf A} over a henselian
field $(K,v)$ of degree {\bf d}. Then by Theorem~\ref{KT3at} there is an
immediate extension $(L|K,v)$ of degree {\bf d}. As $(K,v)$ is
henselian, the extension of $v$ from $K$ to $L$ is unique. Hence by the
Lemma of Ostrowski (cf.\ \cite{End}, \cite{RIB1}),
\[
{\bf d}\>=\>[L:K]\>=\>p^\nu\cdot (vL:vK)\cdot [Lv:Kv] \>=\>p^\nu\;,
\]
where $\nu\in\N\cup\{0\}$ and $p=\chara Kv$. Note that $\nu>0$ because
the degree of {\bf A} is not less than $2$.
\end{proof}

Theorem~\ref{KT2at} and Theorem~\ref{KT3at} together imply:

\begin{proposition}                           \label{i=x}
Every immediate approximation type is realized in some immediate
simple valued field extension.
\end{proposition}

We say that a valued field $(K,v)$ is \bfind{maximal} if it admits no
proper immediate extensions. In this case, by the two theorems, it
admits no immediate approximation types. On the other hand, if $(K,v)$
admits no immediate approximation types, then by part b) of
Lemma~\ref{imme}, it admits no proper immediate extensions. This proves:

\begin{theorem} {\rm\ \ \ (Theorem 4 of \cite{KAP1},
approximation type version)}\label{KT4at}\n
A valued field $(K,v)$ is maximal if and only if
it does not admit immediate approximation types.
\end{theorem}

Similarly, we say that a valued field $(K,v)$ is \bfind{algebraically
maximal} if it does not admit proper immediate algebraic extensions. In
this case, Theorem~\ref{KT3at} shows that it does not admit immediate
algebraic approximation types. On the other hand, if $(K,v)$ admits a
proper immediate algebraic extension $(L|K,v)$, and $x\in L\setminus K$,
then by part b) of Lemma~\ref{imme}, $\appr(x,K)$ is an immediate
approximation type, and by Corollary~\ref{nst}, it is algebraic. This
proves:

\begin{theorem}                             \label{charam}
A valued field $(K,v)$ is algebraically maximal if and only if
it does not admit immediate algebraic approximation types.
\end{theorem}

%
%
\section{The relative approximation degree of polynomials} \label{sectrad}
In view of Proposition~\ref{i=x}, we can from now on assume that every
immediate approximation type {\bf A} is of the form ${\bf A}= \appr
(x,K)$. For the integer $\bh$ that appears in Lemma~\ref{C4+}, where
$\deg f\leq \deg {\bf A}$, we will write $\bh_K(x:f)$ or just
$\bh(x:f)$. We call $\bh(x:f)$ the \bfind{relative approximation
degree of $f(x)$ in $x$} ({\bf over} $K$). From Lemma~\ref{C4+} we know
that
\[
1\leq \bh_K(x:f)\leq \deg f\;.
\]

One can extend the definition of the relative approximation degree to
polynomials of arbitrary degree as follows. Take any polynomial $g\in
K[X]$. Suppose that there exist $\beta\in vK$ and a positive integer $k$
such that
\[
v(g(x)-g(c)) = \beta + k\cdot v(x-c)
\]
for $c \nearrow x$. Note that $\beta$ and $k$ are uniquely determined
because as $\appr(x,K)$ is immediate, there are infinitely many values
$v(x-c)$ for $c\nearrow x$. We will call $k$ the \bfind{relative
approximation degree of $g(x)$ in $x$}, denoted by ${\bf h}_K(x:g)$
as before. Further, we will call $\beta$ the \bfind{relative
approximation constant of $g(x)$ in $x$}, denoted
by\glossary{$\beta_K(x:g)$}
\[
\beta_K (x:g)\;.
\]
By virtue of equation (\ref{bhmin+}) of Lemma~\ref{C4+}, our new
definition of the relative approximation degree coincides with the
definition as given for polynomials of degree at most {\bf d}.
On the other hand, our new definition assigns a relative approximation
degree to every polynomial of arbitrary degree whose value is not fixed,
as Lemma~\ref{deg>d} shows because in this case, $v(g(x)-g(c))=vg(c)$
for $c \nearrow x$. However, for polynomials of degree bigger than
{\bf d}, the relative approximation degree may not be a power of $p$.
Unfortunately, Lemma~\ref{deg>d} does not give information about the
value $v(g(x)-g(c))$ if {\bf A} fixes the value of $g$; this is an open
problem.

\pars
From Lemma~\ref{deg>d} we derive:
\begin{corollary}                           \label{+vgc}
The value of $g$ is fixed by {\bf A} if and only if $vg(x)=vg(c)$ for
$c\nearrow x$. On the other hand, {\bf A} does not fix the value of $g$
if and only if $vg(x)>vg(c)$ for $c\nearrow x$, and this holds
if and only if
\begin{equation}                            \label{betka}
v(g(x)-g(c)) = vg(c)=\beta_{\bf h}(x:g)+{\bf h}_K(x:g)\cdot v(x-c)
\end{equation}
for $c\nearrow x$.
\end{corollary}

For the distances associated with $g(x)$, the following inequalities
will hold in all cases where $\beta_K(x:g)$ and ${\bf h}_K(x:g)$ are
defined:
\begin{equation}                              \label{distieq}
\dist(g(x),K)\geq \dist_K(g(x),g(K))
\geq \beta_K(x:g) + {\bf h}_K(x:g)\cdot \dist(x,K)
\end{equation}
(the first inequality is trivial and the second follows directly from
the definition of relative approximation degree and relative
approximation constant). In the next section, we will consider
various cases where equalities hold.

\parm
We will now investigate the relative approximation degree more closely
for the case of $\deg f\leq \deg {\bf A}$. We will first consider the
relation between ${\bf h}_K(x:f)$ and the approximation type
$\appr(f(x),K)$. Then we show that ${\bf h}_K(x:f)$ is a power of the
characteristic exponent of the residue field, where the
\bfind{characteristic exponent} of a field is defined to be its
characteristic if this is positive, and 1 otherwise. Finally we will
give some hints for the computation of ${\bf h}_K(x:f)$.

\pars
{\bf Throughout this and the next two sections, we will assume
the following situation:}
\begin{equation}                           \label{AF}
\left\{\begin{array}{ll}
\mbox{${\bf A}=\appr(x,K)$} &\mbox{an immediate approximation
type over $(K,v)$}\\
p &\mbox{the characteristic exponent of $Kv$,}\\
{\bf d} &\mbox{the degree of $\appr(x,K)$,}\\
f\in K[X] &\mbox{a nonconstant polynomial of degree
$n\leq \mbox{\bf d}\,$,}\\
{\bf h} &=\bh_K(x:f)\\
\beta_i & \mbox{the fixed value $vf_i(c)$ for $
c\nearrow x$.}
\end{array}\right.
\end{equation}

\pars
\begin{lemma}                     \label{f-g fixed}
Take another polynomial $g\in K[X]$ of degree
at most {\bf d} such that $\appr(x,K)$ fixes the value of $f-g$. If
$\appr(f(x),K) = \appr(g(x),K)$, then ${\bf h}_K(x:f)= {\bf h}_K
(x:g)$ and $\beta_K(x:f)= \beta_K(x:g)$.
\end{lemma}
\begin{proof}
By part b) of Lemma~\ref{v>Lambda}, $\appr(f(x),K) = \appr(g(x),K)$
implies that
\[
v(f(x)-g(x))\geq \dist(f(x),K)\;.
\]
By hypothesis, $\appr(x,K)$ fixes the value of $f-g$, hence by
Lemma~\ref{C4+},
\[
v(f(c)-g(c)) = v(f(x)-g(x)) \geq \dist(f(x),K) \geq v(f(x)-f(c))
\mbox{\ \ for\ \ } c\nearrow x\;.
\]
As (\ref{bhmin+}) shows that the values $v(f(x)-f(c))$ are increasing
for $c\nearrow x$, the last inequality can be replaced by a strict
inequality. So we obtain that
\begin{eqnarray*}
v(g(x)-g(c)) &=& \min\{v(g(x)-f(x))\,,\,v(f(x)-f(c))\,,\,v(f(c)-g(c))\}\\
 &=& v(f(x)-f(c))\>=\> \beta_K(x:f) + {\bf h}_K(x:f)\cdot v(x-c)
\end{eqnarray*}
for $c\nearrow x$. This implies our assertion.
\end{proof}

\pars
To achieve our second goal, we need the following lemma:
\begin{lemma}                      \label{p^t}
If $p$ is prime and $r$ is a positive integer prime to $p$, $r>1$, then
\[
p^t r\choose p^t
\]
is prime to $p$, for every integer $t\geq 0$.
\end{lemma}
\begin{proof}
Consider
\[
{p^t r\choose p^t} = \frac{p^t r (p^t r -1)\cdot\ldots\cdot(p^t r - p^t
+ 1)}{p^t(p^t-1)\cdot\ldots\cdot 1}\;.
\]
In the numerator of this fraction, the first factor $p^t r$ is divisible
by precisely $p^t$, while the remaining factors $p^t r - m$, $1\leq
m\leq p^t-1$, are not divisible by $p^t$. Hence, for every such factor
occurring in the numerator, the corresponding factor $p^t -m=p^t r
-m-p^t(r-1)$ which occurs in the denominator will be divisible by $p$ to
precisely the same power. This gives the desired result.
\end{proof}

Now we are able to prove:

\begin{proposition}                                  \label{L7}
If $i=p^t$, $j=p^t r \leq n$ with $r>1$, $(r,p)=1$,
and if $\beta_i\not=\infty$, then for $c\nearrow x$,
\[
\beta_i + i \cdot v(x-c) < \beta_j + j \cdot v(x-c)\;.
\]
Consequently, ${\bf h}_K(x:f)$ is a power of $p$ (including the case
of ${\bf h}_K(x:f)=1=p^0$).
\end{proposition}
\begin{proof}
We consider the Taylor expansion (\ref{Taylorexpi}) for $f_i(x)$:
\[
\begin{array}{l}
\lefteqn{f_i(x) - f_i(c) =} \\
(i+1)f_{i+1}(c)(x-c) + \ldots
+ {j\choose i} f_j(c)(x-c)^{j-i}+ \ldots
+ {n\choose i} f_n(c)(x-c)^{n-i}\;.
\end{array}
\]
For $c\nearrow x$, the values $vf_{i+1}(c)\, ,\ldots,\,
vf_n(c)$ will be equal to $\beta_{i+1},\ldots,\beta_n$ as defined in
(\ref{AF}). We apply Lemma~\ref{OST} with $m=n-i$, $t_k= k$ for
$1\leq k\leq m$, and
\[
\alpha_1 = v(i+1)+\beta_{i+1}\,,\ldots,\>\alpha_{j-i} =
v{j\choose i}+\beta_{j}\,,\ldots,\>
\alpha_m = v{n\choose i}+\beta_{n}\;.
\]
We find that among the terms on the right hand side of
the Taylor expansion,
there will be precisely one which has least value for $
c\nearrow x$. The value of this term must then equal the value of the
left hand side of the Taylor expansion,
which yields that the latter increases
for $c\nearrow x$. But both values $vf_i(x)$ and $vf_i(c)$
are fixed for $c\nearrow x$. Hence, $v(f_i(x)-f_i(c))>
vf_i(x)=vf_i(c)=\beta_i$ for $c\nearrow x$. It follows that
in particular, the term
\[
{j\choose i} f_j(c)(x-c)^{j-i}
\]
on the right hand side of the Taylor expansion
will also have value $>\beta_i$
for $c\nearrow x$. But $v{j\choose i}=0$: if $p>0$, this is shown in
Lemma~\ref{p^t}, and if $p=1$, then $\chara Kv=0$ which means that
$\chara K=0$ and $v$ is trivial on the subfield $\Q$ of $K$. Therefore,
\[
\beta_i < \beta_j + (j-i)\cdot v(x-c)
\]
for $c\nearrow x$. This yields our assertion.
\end{proof}

%

\pars
The following lemma will give more detailed information on the
computation of ${\bf h}_K(x:f)$.
\begin{lemma}
Assume that $v(x-c)\geq 0$ for $c\nearrow x$. If $i$ is an integer such
that $\beta_i$ is minimal among all $\beta_j$, $j>0$, then ${\bf
h}_K(x:f) \leq i$.
\end{lemma}
\begin{proof}
By assumption, we have that $\beta_j -\beta_i \geq 0$ for all $j>0$.
Further,
\[
\beta_{\bf h} + {\bf h}\cdot v(x-c) < \beta_j + j\cdot v(x-c)
\]
for $j>0$, $j\not= {\bf h}$, and $c\nearrow x$. Thus,
\[
0\leq \beta_{\bf h} - \beta_i \leq (i-{\bf h})\cdot v(x-c)
\]
for $c\nearrow x$, which in view of $v(x-c)\geq 0$ for $c\nearrow x$
yields that $i-{\bf h}\geq 0$, which is the assertion.
\end{proof}

\begin{lemma}                               \label{,h<}
Assume that $p\geq 2$, and write $f(X) = c_n X^n+\ldots+
c_0\,$. Suppose that there exists $i>0$ such that $vc_{i} < vc_k$ for
all $k>0$, $j\not= i$, and write $i = p^t r$ with $r$ prime to $p$. Then
$vf_{\bf h}(c)\geq vc_i$ holds for every $c$ with $vc=0$. And if $vx=0$,
then
\[
{\bf h}_K(x:f) \leq p^t\;.
\]
\end{lemma}
\begin{proof}
For $vc=0$ and $j\geq 1$, by the definition (\ref{iderivative}) of the
$j$-th formal derivative,
\[
vf_j(c)\>=\> v \sum_{k=j}^n {k\choose j} c_k c^{k-j}\>\geq\>
\min_{j\leq k\leq n} v {k\choose j} c_k c^{k-j} \>\geq\> vc_i\;.
\]
By Lemma~\ref{p^t}, the binomial coefficient $p^t r\choose p^t$ is not
divisible by $p$. This shows that $v{p^t r\choose p^t} = 0$ and thus,
\[
vf_{p^t}(c) = vc_i\;.
\]
Now assume in addition that $vx=0$. Then $vc=0$ for $c\nearrow x$.
This yields that
\[
\beta_{p^t} = vc_i\leq \beta_j
\]
for all $j>0$. The foregoing lemma now gives our assertion.
\end{proof}

\begin{corollary}                \label{h(<}
Assume that $vx=0$, and take an integer $e\geq 1$. Suppose
that all nonzero coefficients $c_i$ of $f$, $i>0$, have different values
and that for all $i$ with $p^e | i$, the coefficient $c_i$ is equal to
zero. Then ${\bf h}_K(x:f)<p^e$.
\end{corollary}

%
%
%
\section{Approximation types and distances of polynomials}\label{sectad}
Recall that throughout this section, we assume the situation of
(\ref{AF}).

\begin{lemma}                   \label{LK1a}
The following holds:
\begin{equation}                          \label{xfx}
c\in\appr(x,K)_{\gamma} \Longleftrightarrow f(c)\in
\appr(f(x),K)_{\beta_{\bf h} + {\bf h}\cdot \gamma}
\mbox{ for } c\nearrow x\;.
\end{equation}
In particular,
\begin{equation}                            \label{LK1aeq}
\dist(f(x),K)\geq \dist_K (f(x),f(K))= \beta_{\bf h}+{\bf h}\cdot
\dist(x,K)\;.
\end{equation}
\end{lemma}
\begin{proof}
Equation (\ref{bhmin+}) of Lemma~\ref{C4+} yields (\ref{xfx}),
%
%
while the inequality $\dist(f(x),K)\geq \dist_K(f(x),f(K))$ was already
stated in (\ref{distieq}). It remains to prove that
\[
\dist_K (f(x),f(K)) = \beta_{\bf h}+ {\bf h}\cdot\dist(x,K)\;.
\]
If $\dist(x,K)=\infty$, this equality follows immediately from
(\ref{distieq}). So let us assume from now on that $\dist(x,K)<\infty$.
In order to deduce a contradiction, assume that there exists an element
$c_0\in K$ such that
\[
v(f(x)-f(c_0))> \beta_{\bf h}+{\bf h}\cdot\dist(x,K)\;,
\]
or equivalently,
\[
v(f(x)-f(c_0)) > v(f(x)-f(c))
\]
for $c\nearrow x$. Hence
\begin{eqnarray*}
v(f(c_0)-f(c)) &=& \min\{v(f(x)-f(c)),v(f(x)-f(c_0))\}\\
&=& v(f(x)-f(c)) \>=\> \beta_{\bf h} + {\bf h}\cdot v(x-c)
\end{eqnarray*}
for $c\nearrow x$.
Replacing $x$ by $c_0$ in the Taylor expansion (\ref{Tv}), we find
\begin{eqnarray*}
v(f_1(c_0)\cdot (c-c_0)+ \ldots + f_n(c_0)\cdot (c-c_0)^n)
&=& v(f(c_0)-f(c))\\
&=& \beta_{\bf h} + {\bf h}\cdot v(x-c)
\end{eqnarray*}
for $c\nearrow x$. As noted already at the beginning of
Section~\ref{sectiat}, an immediate approximation type fixes the value
of every linear polynomial. Hence, $v(c-c_0)$ will be fixed for
$c\nearrow x$. On the other hand, the value $\beta_{\bf h} + {\bf h}
\cdot v(x-c)$ is not fixed for $c\nearrow x$, so we conclude that
the value
\[
v(\,f_1(c_0) + f_2(c_0)\cdot (c-c_0)+ \ldots +f_n(c_0)\cdot
(c-c_0)^{n-1}\,)
\]
is not fixed for $c\nearrow x$. This proves the existence of a
polynomial of degree $n-1$ whose value is not fixed by $\appr(x,K)$. But
$n-1=\deg f -1<{\bf d}$, a contradiction. This proves the desired
equality.
\end{proof}

\begin{lemma}                         \label{LK1b}
Assume that $\deg f< \mbox{\bf d}$. Then $\appr(f(x),K)$ is
an immediate approximation type over $K$ with
\begin{equation}                 \label{distEq}
\dist(f(x),K) = \dist_K (f(x),f(K)) = \beta_{\bh}+{\bh}\cdot\dist(x,K),
\end{equation}
and $\appr(f(x),K)$ is determined by (\ref{xfx}).
\end{lemma}
\begin{proof}
%
In view of (\ref{LK1aeq}), to prove the first equality in (\ref{distEq})
we have to show that for every $b\in K$ there exists an element $c\in K$
such that $v(f(x)-f(c))\geq v(f(x)-b)$. Since $\deg(f-b)= \deg f
<\mbox{\bf d}$, it follows that $\appr(x,K)$ fixes the value of $f-b$.
Applying Lemma~\ref{C4+} to $f-b$ in place of $f$, we deduce that
$v(f(x)-b) = v(f(c)-b)$ for $c\nearrow x$. Consequently, for such an
element $c\in K$ we get that
\[
v(f(x)-f(c))\geq\min\{v(f(x)-b),v(f(c)-b)\} = v(f(x)-b)\;,
\]
as desired.

By the second equality of (\ref{distEq}), which has already been proved
in Lemma~\ref{LK1a}, we know that there exists $c'\in K$ such that
$v(f(x)-f(c'))>v(f(x)-f(c))\geq v(f(x)-b)$. We have proved that for
every $b\in K$ there is $b'=f(c')\in K$ such that $v(f(x)-b')>
v(f(x)-b)$. Part a) of Lemma \ref{imme} now shows that $\appr(f(x),K)$
is immediate.

By (\ref{distEq}), the values $\beta_{\bh}+{\bh}\cdot v(x-c)$ are
cofinal in $\supp\, \appr(f(x),K)$ for $c\nearrow x$. Therefore,
$\appr(f(x),K)$ is determined by the balls $\appr(f(x),K)_{\beta_{\bh}
+{\bh}\cdot v(x-c)}$ for those $c$, which in turn are determined by
(\ref{xfx}).
\end{proof}

\begin{corollary}                               \label{+gf}
Assume that $\deg f< \mbox{\bf d}$, and let $d'\geq 1$ be a
natural number such that $d'\cdot \deg f\leq\mbox{\bf d}$. Then
\[
\deg \appr(f(x),K) \geq d'\;.
\]
In particular, if $\appr(x,K)$ is transcendental, then so is
$\appr(f(x),K)$.
\end{corollary}
\begin{proof}
Take a polynomial $g$ of degree smaller than $d'\leq {\bf d}$. Suppose
that\linebreak
$\appr(f(x),K)$ does not fix the value of $g$. Then by Lemma~\ref{C4+},
\[
vg(f(x))\> >\> vg(a)
\]
for $a\nearrow f(x)$. Since $\deg f <\mbox{\bf d}$, Lemma~\ref{LK1b}
shows that $\dist_K (f(x),f(K)) = \dist(f(x),K)$, so
\[
vg(f(x))\> >\> vg(f(c))
\]
for $c\nearrow x$. But then by Lemma~\ref{C4+}, $\appr(x,K)$ does
not fix the value of the polynomial $f(g(X))$. This contradicts the fact
that its degree is smaller than {\bf d}.
\end{proof}

\begin{lemma}                
Assume that $\appr(x,K)$ does not fix the value of $f$ (hence $\deg f
=\mbox{\bf d}$). Then
\[
vf(x) > \beta_{\bf h} + {\bf h}\cdot v(x-c) \quad\mbox{for $c\nearrow x$.}
\]
\end{lemma}
\begin{proof}
We rewrite (\ref{Tv}) as follows:
\[
-f(c) = f_1(c)\cdot (x-c)+\ldots+ f_n(c)\cdot (x-c)^n - f(x)\;.
\]
Suppose
that $vf(x)<\beta_{\bf h} + {\bf h}\cdot v(x-c)$ for $c\nearrow x$. This
in turn implies that the value of the right hand side is equal to $vf(x)$
and hence the value $vf(c)$ is fixed for for $c\nearrow x$, which
contradicts our assumption. This proves that $v f(x) \geq \beta_{\bf h}
+ {\bf h}\cdot v(x-c)$, and since
$v(x-K)$ has no maximal element, also
$vf(x)>\beta_{\bf h} + {\bf h}\cdot v(x-c)$ for $c\nearrow x$.
\end{proof}

Note that in the case of $\deg f =\mbox{\bf d}$ we can only say
that ``$\appr(f(x),K)$ is determined by (\ref{xfx})
up to $\dist_K (f(x),f(K))$''. But it may happen that
\[
\dist(f(x),K) > \dist_K(f(x),f(K))\;.
\]
This will usually be the case when $f$ is the minimal polynomial of $x$,
which yields that $f(x)=0$ and hence $\dist(f(x),K)=\dist(0,K)=\infty$.

\begin{example}
Take $(L,v)$ and $f(X)= X^p - X - t^{-1}$ with root $\vartheta$ as in
Example~\ref{d<degree}. As noted there, $v(\vartheta-L)=\{\alpha\in vL
\mid\alpha<0\}$, so $\dist(\vartheta,L)$ is the cut in $\widetilde{vL}$
whose lower cut set consists of all negative elements. This implies
that $\dist(\vartheta,L)=p\cdot \dist(\vartheta,L)$.

We have that $f(X)-f(c)=X^p-X-(c^p-c)=(X-c)^p-(X-c)$. Since
$v(\vartheta-c)<0$, it follows that $v(\vartheta-c)^p=p\cdot
v(\vartheta-c)<v(\vartheta-c)$ and therefore, $v(f(\vartheta)-f(c))=
v((\vartheta-c)^p-(\vartheta-c))= \min\{v(\vartheta-c)^p,
v(\vartheta-c)\}=p\cdot v(\vartheta-c)$.
%
%
This shows that $\bh_L(x:f)=p$ and $\beta_L(x:f)=0$.
We obtain that
\[
\dist(f(\vartheta),L)\>=\>\infty\> >\>\dist(\vartheta,L)\>=\>p\cdot
\dist(\vartheta,L)\>=\>\dist_L(f(\vartheta), f(L))\>,
\]
where the last equality holds by Lemma~\ref{LK1a}.
\end{example}

%
%
%
\section{The degree $[K(x)^h:K(f(x))^h]$}         \label{Atd,rad}
In the situation of (\ref{AF}), we ask for the degree
\[
[K(x)^h:K(f(x))^h]\;.
\]
This can indeed be calculated by means of ${\bf h}_K(x:f)$.
Inequality (\ref{,7eq}) below will explain the origin of the notation
``${\bf h}_K(x:f)$''. Note that $[K(x):K(f(x))]=\deg f$, while in
general, we may have that $[K(x)^h:K(f(x))^h]<\deg f$.

\begin{theorem}                                   \label{,7}
Assume (\ref{AF}). Then
\begin{equation}                            \label{,7eq}
[K(x)^h:K(f(x))^h] \leq {\bf h}_K(x:f)\;.
\end{equation}
\end{theorem}
\begin{proof}
We consider the Taylor expansion (\ref{Taylorexp}) of $f$ for an
arbitrary $c\in K$.
From Lemma~\ref{C4+}, we know that (\ref{xf-}) holds for
$1\leq i\leq \deg f$, $i\not={\bf h}={\bf h}_K(x:f)$ and $c\nearrow x$.
We choose such an element $c\in K$ and also an element $d\in K$ with
$vd=-v(x-c)$. We set $x_0=d\cdot (x-c)$; hence $vx_0=0$ and $K(x)=K(x_0)$.
Now (\ref{xf-}) takes the form
\begin{equation}                            \label{vcoeff1}
v(f_i(c)d^{-i}) > v(f_{\bf h}(c)d^{-{\bf h}})\;\;\mbox{\rm for
$i\not={\bf h}$, $1\leq i\leq \deg f$,}
\end{equation}
and (\ref{bhmin+}) reads as
\begin{equation}                            \label{vcoeff2}
v(f(x)-f(c)) = vf_{\bf h}(c)d^{-{\bf h}}\;.
\end{equation}
Further, from (\ref{Taylorexp}), (\ref{vcoeff1}) and (\ref{vcoeff2}) we
obtain:
\begin{eqnarray}                            \label{vcoeff3}
\left(\frac{d^{\bf h}}{f_{\bf h}(c)}\cdot (f(c)-f(x))\right)v &=&
\left(-\frac{d^{\bf h}}{f_{\bf h}(c)}\cdot
\sum_{i=1}^{\mbox{\scriptsize\rm deg} f} f_i(c)(x - c)^i \right)v\\
& = &\left(-\sum_{i=1}^{\mbox{\scriptsize\rm deg} f}
\frac{f_i(c)d^{-i}}{f_{\bf h}(c)d^{-{\bf h}}}\,  x_0^i\right)v
\> = \> -(x_0v)^{\bf h}\;.      \nonumber
\end{eqnarray}

Now we set
\[
\tilde{f}(Z)=\sum_{i=0}^{\mbox{\scriptsize\rm deg} f} f_i(c)d^{-i}
Z^i\;;
\]
hence $\tilde{f}(x_0)= f(x)$. Let us consider the polynomial
\[
F(Z) = \frac{d^{\bf h}}{f_{\bf h}(c)}\cdot (\tilde{f}(Z) -
\tilde{f}(x_0))
\]
whose coefficients lie in $K(\tilde{f}(x_0))=K(f(x))$ and for which
$x_0$ is a zero. Using (\ref{vcoeff1}) and (\ref{vcoeff2}), we compute
\[
F(Z) = \frac{d^{\bf h}}{f_{\bf h}(c)}\cdot (f(c)-f(x)) +
\sum_{i=1}^{\mbox{\scriptsize\rm deg} f}
\frac{f_i(c)d^{-i}}{f_{\bf h}(c)d^{-{\bf h}}}\, Z^i\in
{\cal O}_{K(f(x))}[Z]
\]
and, using also (\ref{vcoeff3}),
\[
F(Z)v = Z^{\bf h} - (x_0v)^{\bf h} = (Z - x_0v)^{\bf h}
\]
(where the latter equation holds because by Proposition~\ref{L7},
${\bf h}$ is a power of $p$). Using the strong Hensel's Lemma, that is,
property 3) of Theorem~4.1.3 in \cite{[Eng--P]}, we deduce that there
is a factorization
\[
F(Z) = G(Z)H(Z)
\]
over $K(f(x))^h$ with
\[
G(Z)v = Z^{\bf h} - (x_0v)^{\bf h}
\]
and
\[
\deg G(Z) = \deg G(Z)v = {\bf h}\;.
\]
A zero of $F(Z)$ which has residue $x_0v$ cannot be a zero of
$H(Z)$ since $H(Z)v = 1$, hence it must appear as a zero of $G(Z)$.
In particular, $G(x_0) = 0$. Since $G(Z) \in K(f(x))^h[Z]$ and
$\deg G(Z) = {\bf h}$, and since $K(x_0)^h=K(f(x))^h(x_0)$, this shows
that
\[
[K(x)^h:K(f(x))^h]=[K(x_0)^h:K(f(x))^h]\leq {\bf h}={\bf h}_K(x:f)\;.
\]
\end{proof}

\begin{corollary}
In addition to (\ref{AF}), assume that $(K,v)$ is henselian and $x$ is
algebraic over $K$. If $\mbox{\bf d}=[K(x):K]$ and $f$ is the minimal
polynomial of $x$ over $K$, then $p\geq 2$ and
\[
[K(x):K]={\bf h}_K(x:f)=p^t
\]
for some integer $t\geq 1$.
\end{corollary}
\begin{proof}
By hypothesis, we have $\mbox{\bf d}=[K(x):K]=\deg f$. Since $K$ is
henselian and $x$ is algebraic over $K$, we have that $K(x)$ is
henselian as well. In view of $f(x)=0$, an application of the foregoing
lemma shows that
\[
\deg f = [K(x):K] \leq {\bf h}_K(x:f) \leq \deg f\;.
\]
Consequently, equality holds everywhere.

Since $\appr(x,K)$ is immediate by assumption, it is nontrivial, hence
$x\notin K$ and ${\bf h}_K(x:f)=[K(x):K]> 1$. Proposition~\ref{L7}
yields that $p\geq 2$ and ${\bf h}_K(x:f)=p^t$ with $t\geq 1$.
\end{proof}

%
%
\section{The degree $[K(x)^h:K(y)^h]$}   \label{,ad2}
Throughout this section, we will work with the following situation:
\begin{equation}                  \label{,sit1}
\left\{\begin{array}{l}
\mbox{$(K,v)$ a valued field of rank 1}\\
\mbox{$(K(x)|K,v)$ an immediate extension such that $x\notin K^c$}\\
\mbox{and $\appr (x,K)$ is transcendental}\\
\mbox{$y\in K(x)^h$ transcendental over $K$.}
\end{array}\right.
\end{equation}
Note that by Corollary~\ref{KT2at}, the assumption that $\appr (x,K)$ is
transcendental implies that $x$ is transcendental over $K$. Furthermore,
if $(K,v)$ is algebraically maximal, then $\appr (x,K)$ is always
transcendental, provided that $(K(x)|K,v)$ is immediate and nontrivial.

We ask for the degree
\[
\;[K(x)^h:K(y)^h]\;.
\]
To treat this question and in particular to define the relative
approximation degree of $x$ over $y$, we look for a
polynomial $f\in K[X]$ such that
\begin{equation}                            \label{aty=atf}
\;v(y-f(x))\geq \dist(y,K)\;.
\end{equation}
We need some preparation.

\begin{lemma}                               \label{,1id}
If $K$ is of rank 1 and $K(x)|K$ is immediate, then $K[x]$ is dense
in $K(x)^h$.
\end{lemma}
\begin{proof}
Since any valued field of rank 1 is dense in its henselization, it
suffices to show that $K[x]$ is dense in $K(x)$. For this we only have
to show that for every $f(x)\in K[x]$ and every $\alpha\in vK$ there
exists an element $g(x)\in K[x]$ such that $v(g(x) - 1/f(x))> \alpha$.
Since $K(x)|K$ is immediate there is an element $c\in K$ satisfying
$v(c-f(x))>vf(x) = vc$, which yields that $v(1-f(x)/c) > 0$. By our
hypothesis on the rank which means that the value group $vK$ is
archimedian, there exists $j\in \N$ such that $j\cdot v(1-f(x)/c)>\alpha
+ vc$. Now we put $h(x) = 1-f(x)/c \in K[x]$ and compute
\begin{eqnarray*}
v\left(\frac{1}{f(x)}-c^{-1}\sum_{i=0}^{j-1}h(x)^i\right) & = &
  v\left(\frac{1}{c(1 - h(x))}-c^{-1}\sum_{i=0}^{j-1}h(x)^i\right)\\[.3cm]
& = & vc^{-1}h(x)^j = j\cdot v(1-f(x)/c) - vc > \alpha\;.
\end{eqnarray*}
As the sum is an element of $K[x]$, this proves our lemma.
\end{proof}

\begin{lemma}                      \label{Laty=atf}
Assume (\ref{,sit1}). Then $y\in K[x]^c\setminus K^c$ and there exists a
polynomial $f\in K[X]$ such that (\ref{aty=atf}) holds.
\end{lemma}

\begin{proof}
From Lemma~\ref{,1id}, we infer that $y\in K[x]^c$. Suppose that
$y\in K^c$. Then $K$ is dense in $K(y)$ and also in $K(y)^h$ since
$K(y)$ is dense in its henselization, being of rank 1 like $K$. Let
$g(X)\in K(y)^h[X]$ be the minimal polynomial of $x$ over $K(y)^h$. We
can choose polynomials $\tilde{g}(X)\in K[X]$ with coefficients
arbitrarily close to the corresponding coefficients of $g$. By the
continuity of roots (cf.\ Theorem~4.5 of [PZ]) and our assumption that
$x\notin K^c$, i.e., $\dist(x,K)<\infty$, we can find a suitable
polynomial $\tilde{g}$ with a suitable root $\tilde{x}\in \tilde{K}$
such that
\[
v(x-\tilde{x}) \geq\dist(x,K)\;.
\]
By Lemma~\ref{v>Lambda} b), this implies that
\[
\appr(x,K) = \appr(\tilde{x},K)\;.
\]
Since $\tilde{x}$ is algebraic over $K$, it follows by Corollary~\ref{nst}
that $\appr(\tilde{x},K)$ and hence $\appr(x,K)$ is an algebraic
approximation type over $K$, a contradiction to hypothesis
(\ref{,sit1}). This shows that $y\notin K^c$, i.e., $\dist(y,K)<\infty$.
As $y\in K[x]^c$, this shows the existence of a polynomial $f\in K[X]$
such that $v(y-f(x))\geq\dist(y,K)$.
\end{proof}

With $f$ as in this lemma, we define \glossary{{\bf h}$_K(x:y)$}
\[
\bh_K(x:y) \>:=\> \bh_K(x:f) \;\mbox{ and }\;
\beta_K(x:y) \>:=\> \beta_K(x:f)
\]
and call $\bh_K(x:y)$ the \bfind{relative approximation degree of
$y$ in $x$}\ (over $K$).

\begin{lemma}                               \label{bhwd}
The integers
$\bh_K(x:y)$ and $\beta_K(x:y)$ are well-defined, i.e., they do not
depend on the choice of $f(x)$ as long as $v(y-f(x)) \geq \dist(y,K)$ is
satisfied.
\end{lemma}
\begin{proof}
If $g(x)$ is another polynomial in $K[x]$ such that
$v(y-g(x)) \geq \dist(y,K)$, then by Lemma~\ref{v>Lambda}, we have that
\[
\appr(g(x),K) = \appr(y,K) = \appr(f(x),K)\;,
\]
whence $\bh_K(x:g)=\bh_K(x:f)$ and $\beta_K(x:g)=\beta_K(x:f)$
by Lemma~\ref{f-g fixed} since $\appr(x,K)$ is transcendental.
\end{proof}

\pars
In the situation described in (\ref{,sit1}), we can prove
Theorem~\ref{,7} also for $y$ in place of $f(x)$ provided that the
extension $K(x)^h | K(y)^h$ is separable. For the proof, we need the
following lemma:

\begin{lemma}                     
Assume (\ref{,sit1})
and let $v(y-f(x)) \geq \dist(y,K)$. Then there exists an element $z$ in
the algebraic closure $\widetilde{K(y)}$ of $K(y)$ such that
\[
[K(y,z)^h:K(y)^h]\leq \bh = \bh_K(x:y)
\]
and
\[
v(x-z)\geq \frac{1}{\bh}\left(v(y-f(x)) - \beta_K(x:f)\right)\;.
\]
\end{lemma}
\begin{proof}
Recall that $\bh = \bh_K(x:y)=\bh_K(x:f)$. We put $r:= y-f(x)$. We
choose $c,d\in K$, $x_0$ and $F(Z)$ as in the proof of Theorem~\ref{,7}.
Then
\begin{eqnarray*}
v r \>\geq\> \dist(y,K) &>& v(y-f(c)) \>=\> v(f(x)-f(c))\\
&=& v(f_{\bhsc} (c)(x - c)^{\bhsc}) \>=\> v(f_{\bhsc} (c)d^{-\bhsc})\;.
\end{eqnarray*}
This shows that
\[
F^{\circ}(Z):= F(Z) - \frac{d^{\bhsc}}{f_{\bhsc}(c)}\cdot r
= \frac{d^{\bhsc}}{f_{\bhsc}(c)}\cdot (\tilde{f}(Z) - y) \>\in\>
{\cal O}_{K(y)}[Z]
\]
has the same reduction as $F(Z)$. We find, as for $F(Z)$, that
$F^{\circ}(Z)$ admits a factorization
\[
F^{\circ}(Z) = G^{\circ}(Z)H^{\circ}(Z)
\]
over $K(y)^h$ with $G^{\circ}(Z)v = Z^{\bhsc} - (x_0v)^{\bhsc}$,
$G^{\circ}$ monic, $\deg G^{\circ}(Z) = \deg G^{\circ}(Z)v = \bh$
and $H^{\circ}(Z)v = 1$. Note that $vF^{\circ}(x_0) = vG^{\circ}(x_0)$
since $x_0\in {\cal O}_{K(x)}$. Recall that $F(x_0)=0$. Consequently,
from
\[
F^{\circ}(x_0) = -\frac{d^{\bhsc}}{f_{\bhsc}(c)}\cdot r
\]
it follows that, with $\beta_{\bhsc} =vf_{\bhsc}(c)=\beta_K(x:f)$,
\[
v(d^{\bhsc} r) - \beta_{\bhsc} = vF^{\circ}(x_0) =
vG^{\circ}(x_0)\;.
\]
Hence there must exist a root $z_{j_0}$ of
\[G^{\circ}(Z)=\prod_{1\leq j\leq \bhsc} (Z-z_j)\;,\;\; z_j\in
\widetilde{K(y)}\]
with
\[
v(x_0-z_{j_0})\geq \frac{1}{\bh} \left(v(d^{\bhsc} r)
-\beta_{\bhsc}\right)\;,
\]
which is equivalent to
\[
v(x-(d^{-1}z_{j_0} + c))\geq
\frac{1}{\bh} \left(vr -\beta_{\bhsc}\right)=
\frac{1}{\bh}\left(v(y-f(x)) - \beta_K(x:f)\right)\;.
\]
Now $z:=d^{-1}z_{j_0} + c$ is the element of our assertion, since
it satisfies $K(y,z) = K(y,z_{j_0})$ and thus $[K(y,z)^h:K(y)^h]\leq\bh$.
\end{proof}

\begin{proposition}                                \label{,6prop}
Assume (\ref{,sit1}).
If $K(x)^h|K(y)^h$ is separable, then
\[
[K(x)^h:K(y)^h] \>\leq\> \bh_K(x:y)\;.
\]
\end{proposition}
\begin{proof}
Set
\[
\alpha:=\max\{v(\sigma x - x)\mid\sigma\in\mbox{Gal}(\widetilde{K(y)}
|K(y)^h)\mbox{ with } \sigma x\not= x\}\;.
\]
Then $\alpha<\infty$ since $K(x)^h|K(y)^h$ is separable.
Now, by Lemma~\ref{Laty=atf} we can choose $f(x)\in K[x]$ such that
$v(y-f(x))\geq\dist(y,K)=\dist(f(x),K)$ as well as
\[
v(y-f(x)) \> >\> \beta_K(x:y) + \bh\alpha\>=\>\beta_K(x:f) +
\bh\alpha\;,
\]
where $\bh = \bh_K(x:y)$.
Using the foregoing lemma, we choose $z\in \widetilde{K(y)}$ such that
\[
v(x-z)\>\geq\>\frac{1}{\bh}\left(v(y-f(x)) - \beta_K(x:f)\right)
\> >\>\alpha\;,
\]
and $[K(y,z)^h:K(y)^h]\leq \bh$. In view of our separability condition,
we can deduce by Krasner's Lemma (see \cite{[Eng--P]}, Theorem~4.1.7)
that $x\in K(y)^h (z)$. This yields that $[K(x,y)^h:K(y)^h]\leq
[K(y,z)^h:K(y)^h]\leq \bh$. Since $y\in K(x)^h$ by assumption,
$K(x,y)^h=K(x)^h$ and thus $[K(x)^h:K(y)^h]\leq\bh$, as asserted.
\end{proof}

In order to prove the assertion of the proposition without the
separability condition, we need the following tool.

\begin{lemma}                     \label{,tr}
Assume that (\ref{,sit1}) holds. Then it also holds for $y$
in place of $x$. So if $z\in K(y)^h$ is transcendental over $K$,
then $\bh_K(y:z)$ is defined. In this situation, $\bh_K(x:z)=
\bh_K(x:y)\cdot\bh_K(y:z)$.
\end{lemma}
\begin{proof}
Recall that from Lemma~\ref{Laty=atf} we have that $y \notin K^c$.
Moreover, as $K(y)|K$ is a subextension
of the immediate extension $K(x)^h|K$, it is also
immediate. For the definition of $\bh_K(x:y)$ we have already used the
fact that there exists some polynomial $f(x)$ such that $\appr(y,K) =
\appr(f(x),K)$; by Corollary~\ref{+gf}, this approximation type is
transcendental since $\appr(x,K)$ is. We have proved that (\ref{,sit1})
holds for $y$ in place of $x$.

Let us now prove the multiplicativity. Since $\bh_K(y:z)=\bh_K(y:g(y))$
whenever $v(z-g(y))\geq\dist(z,K)$, it suffices to show our assertion
under the additional assumption $z=g(y)\in K[y]$. Furthermore, because
of $y\in K[x]^c\setminus K^c$ we may choose $f(x)\in K[x]$ so that
$v(y-f(x))\geq\dist(y,K)$ and $v(g(y)-g(f(x)))\geq\dist(g(y),K)$; hence
it suffices to show our assertion under the assumption that $y=f(x)\in
K[x]$ and $z=g(f(x))\in K[x]$. Since by hypothesis, $\appr(x,K)$ is
transcendental, it fixes the value of every polynomial over $K$, and
thus we know from Lemma~\ref{LK1b} that $f(c)\nearrow f(x)$ whenever
$c\nearrow x$. Also since $\appr(f(x),K)$ is transcendental, it fixes
the value of every polynomial over $K$, and thus for $f(c)\nearrow
f(x)$,
\begin{eqnarray*}
v(g(f(x))-g(f(c))) &=& v g_{\bhsc_1}(f(c)) + \bh_1\cdot v(f(x)-f(c))\\
&=& v g_{\bhsc_1}(f(c)) + \bh_1\cdot\left(v f_{\bhsc_2}(c)
+ \bh_2 \cdot v(x-c)\right)\\
&=& \beta + \bh_1\cdot \bh_2\cdot v(x-c)
\end{eqnarray*}
where $\bh_1=\bh_K(f(x):g(f(x)))$, $\bh_2=\bh_K(x:f)$ and
$\beta=v g_{\bhsc_1}(f(c)) + \bh_1\cdot v f_{\bhsc_2}(c)$. This shows
that
\[
\bh_K(x:g(f(x))) = \bh_1\cdot \bh_2 = \bh_2\cdot \bh_1
=\bh_K(x:f)\cdot\bh_K(f(x):g(f(x)))\;,
\]
as asserted.
\end{proof}

\begin{theorem}                             \label{,6}
Assume (\ref{,sit1}). Then
\[
[K(x)^h:K(y)^h] \>\leq\> \bh_K(x:y)\;.
\]
\end{theorem}
\begin{proof}
Take $p^n$ to be the inseparable degree of $K(x)^h|K(y)^h$ and $L|K(y)^h$
to be the maximal separable subextension of $K(x)^h|K(y)^h$. Then
$[K(x)^h:L]=p^n$. Further, $x^{p^n}$ is separable over $K(y)^h$, so
$x^{p^n}\in L$ and $K(x^{p^n})^h\subseteq L$. As
$K(x)^h=K(x^{p^n})^h(x)$, we find that
\[
p^n\>\geq\> [K(x)^h:K(x^{p^n})^h] \>=\> [K(x)^h:L]\cdot [L:K(x^{p^n})^h]
\>=\> p^n\cdot [L:K(x^{p^n})^h]\>,
\]
which shows that $[L:K(x^{p^n})^h]=1$ and in particular, $y\in
K(x^{p^n})^h$. So we are able to apply Lemma~\ref{,tr} to obtain that
$\bh_K(x:y)=\bh_K(x:x^{p^n})\cdot\bh_K(x^{p^n}:y)=p^n\cdot
\bh_K(x^{p^n}:y)$.

As $x^{p^n}$ is separable over $K(y)^h$, we can infer from
Proposition~\ref{,6prop} that $[K(x^{p^n})^h:K(y)^h]\leq\bh_K(x^{p^n}:y)$.
On the other hand, $[K(x)^h:K(x^{p^n})^h]=[K(x)^h:L]=p^n$. So we get
\[
[K(x)^h:K(y)^h]\>=\>p^n\cdot [K(x^{p^n})^h:K(y)^h]\>\leq\>
p^n\cdot \bh_K(x^{p^n}:y)\>=\>\bh_K(x:y)\>,
\]
as desired.
\end{proof}

\begin{corollary}                       \label{,11b}
Assume that (\ref{,sit1}) holds. Then
\[
K(x)^h = K(y)^h \>\Longleftrightarrow\> \bh_K(x:y)=1\;.
\]
\end{corollary}
\begin{proof}
If $K(x)^h = K(y)^h$, then $x\in K(y)^h$ and $y\in K(x)^h$, and by
Lemma~\ref{,tr} we have that
\[
\bh_K(x:y)\cdot\bh_K(y:x)\>=\>\bh_K(x:x)\>=\>1\;,
\]
which yields $\bh_K(x:y)=1$. The reverse implication follows from
Theorem~\ref{,6}.
\end{proof}

%
%
\section{An application to henselian rationality}
In this section we will apply Theorem~\ref{,6} to immediate valued
function fields which are the henselization of a rational function
field.

\begin{theorem}                                    \label{hr-down}
Take a valued field $(K,v)$ of rank 1 and an immediate function field
$(F|K,v)$ of transcendence degree 1. Suppose there is some $x\in
F^h\setminus K^c$ with transcendental approximation type over $K$ such
that $F^h=K(x)^h$. Then there is already some $y\in F$ such that
$F^h=K(y)^h$. In fact, there is some $\gamma\in vK$ such that
$K(x)^h=K(y)^h$ holds for every $y\in F$ with $v(x-y)\geq\gamma$.
\end{theorem}
\begin{proof}
Since $x\notin K^c$ there is $\gamma \in vK$ such that $\gamma>
\dist(x,K)$. By assumption, the rank of $(K,v)$ is 1, and since
$(F|K,v)$ is immediate, also $(F,v)$ has rank 1. Thus, the element
$x$ lies in the completion of $F$. So we may take some $y\in F$
such that $v(x-y)\geq\gamma>\dist (x,K)$. For every such $y$,
$[K(x)^h:K(y)^h] \leq \bh_K(x:y)$ holds by Theorem~\ref{,6}, and
$\bh_K(x:y)=\bh_K(x:x)=1$ holds by Lemma~\ref{bhwd}. This yields
that $K(x)^h=K(y)^h$.
\end{proof}

%
%
\section{Approximation coefficients}
Throughout this section, we will assume the situation as described in
(\ref{,sit1}). As before, take $f(x)\in K[x]$ such that
$v(y-f(x))\geq\dist(y,K)$. An element $d\in K$ will be called an
\bfind{approximation coefficient of $y$ in $x$}\ (over $K$), if
\begin{equation}                 \label{,ac}
v(f(x) - f(c)) < v(f(x)-f(c)-d\cdot (x-c)^{\bhsc})
\end{equation}
for $c\nearrow x$, where $\bh=\bh_K(x:y)$.

\begin{lemma}                     
If $d$ satisfies (\ref{,ac}) for some $f(x)$ with
$v(y-f(x))\geq\dist(y,K)$,
then it satisfies (\ref{,ac}) for every such $f(x)$; in other words:
approximation coefficients are independent of the choice of $f(x)$. If
$d$ satisfies (\ref{,ac}), then it satisfies
\begin{equation}                 \label{,ac'}
v(y - f(c)) <v(y-f(c)-d\cdot (x-c)^{\bhsc})
\;\;\;\mbox{ for } c\nearrow x\;.
\end{equation}
\end{lemma}
\begin{proof}
If $g(x)$ is another element of $K[x]$ with $v(y-g(x))\geq\dist(y,K)$,
then
\[
v(f(x)-g(x))\geq\dist(y,K)=\dist(f(x),K) > v(f(x)-f(c))
\]
for all $c\in K$. Since $\appr(x,K)$ is transcendental, it fixes the
value of the polynomial $f-g$, whence
\[
v(f(c)-g(c))=v(f(x)-g(x))>v(f(x)-f(c))
\;\;\;\mbox{ for } c\nearrow x\;.
\]
Hence by the ultrametric triangle law,
\begin{eqnarray*}
v(g(x)-g(c)) & = & \min\{v(g(x)-f(x)),v(f(x)-f(c)),v(f(c)-g(c))\}\\
 & = & v(f(x)-f(c))
\end{eqnarray*}
and
\begin{eqnarray*}
\lefteqn{v(g(x)-g(c)-d\cdot (x-c)^{\bhsc})}\\
&\geq& \min\{v(f(x)-f(c)-d\cdot (x-c)^{\bhsc})\,,\>
v(f(x)-g(x))\,,\>v(f(c)-g(c))\}\\
&>& v(f(x)-f(c))=v(g(x)-g(c))
\end{eqnarray*}
for $c\nearrow x\,$,
which shows that $d$ fulfills equation (\ref{,ac}) also with $g$ in
place of $f$. Replacing $g(x)$ by $y$ and $g(c)$ by $f(c)$ in
the above deduction, one obtains a proof of (\ref{,ac'}).
\end{proof}

The following lemma proves the existence of approximation coefficients:

\begin{lemma}                     \label{exapco}
The element $d\in K$ is an approximation coefficient of $y$ in $x$ if
and only if
\[
vd = v f_{\bhsc}(c) < v(f_{\bhsc}(c) - d) \mbox{\ \ for\ \ }c\nearrow x\;.
\]
In particular, there exists an approximation coefficient of $y$ in $x$.
Furthermore,
\begin{equation}                                     \label{,dist_d}
\dist(y,K)=vd+\bh\cdot \dist(x,K)
\end{equation}
\end{lemma}
\begin{proof}
By definition of $\bh=\bh_K(x:y)=\bh_K(x:f)$, we have that
\[
v(f(x)-f(c) - f_{\bhsc}(c)(x-c)^{\bhsc}) > v(f(x)-f(c))
= v(f_{\bhsc}(c)(x-c)^{\bhsc})
\]
for $c\nearrow x$. Hence (\ref{,ac}) holds
for $c\nearrow x$ if and only if
\[
v(f_{\bhsc}(c)(x-c)^{\bhsc}-d\cdot (x-c)^{\bhsc}) >
v(f_{\bhsc}(c)(x-c)^{\bhsc})\>,
\]
which is equivalent to
\[
v f_{\bhsc}(c)  < v(f_{\bhsc}(c)-d)\mbox{\ \ for\ \ }c\nearrow x\;.
\]
Since $K(x)|K$ is assumed to be an immediate extension, by
Lemma~\ref{imme} a) there exists some $d\in K$ such that
$v(f_{\bhsc}(x)-d)>v f_{\bhsc}(x)$. Since $\appr(x,K)$ is
transcendental, for $c\nearrow x$ we have that
$v(f_{\bhsc}(c)-d)= v(f_{\bhsc}(x)-d)$ and $v f_{\bhsc}(c)=
vf_{\bhsc}(x)$ and thus,
\[
v(f_{\bhsc}(c)-d)=v(f_{\bhsc}(x)-d)>v f_{\bhsc}(x)=v f_{\bhsc}(c)=vd\;.
\]
Hence $d$ is an approximation coefficient for $y$ in $x$ by the first part
of our proof.

In view of the hypothesis that $ \appr(x,K) $ is transcendental, $f(x)$
satisfies equation~(\ref{distEq}) of Lemma~\ref{LK1b}. From this we
obtain:
\begin{eqnarray*}
\dist(y,K)&=&\dist(f(x),K)=vf_{\bh}(c) + \bh\cdot\dist(x,K)\\
&=& vd + \bh\cdot\dist(x,K)\;.
\end{eqnarray*}
\end{proof}

\begin{lemma}                     \label{,acm}
Take elements $y_i\in K[x]^c\setminus K^c$ with common approximation
degree $\bh=\bh_K(x:y_i)$, $1\leq i\leq m$. Assume that $d_i\in K$ is an
approximation coefficient of $y_i$ in $x$ and let $k_i$ be elements in
$K$ such that
\begin{equation}                              \label{,ack}
v\sum_{i=1}^m k_i d_i \,= \min_{1\leq i \leq m} vk_i d_i\> <\infty\;.
\end{equation}
Then the following will hold:
\[
\bh_K\left(x:\sum_{i=1}^m k_i y_i\right) = \bh\;.
\]
\end{lemma}
\begin{proof}
We choose polynomials $f^{[i]}(X)\in K[X]$ with $v(y_i-f^{[i]}(x)) \geq
\dist(y_i,K)$. Then by Lemma~\ref{v>Lambda} b), we have that
$\dist(f^{[i]}(x),K) = \dist(y_i,K)$. We set
\[
g(X):=\sum_{i=1}^m k_i f^{[i]}(X) \in K[X]
\]
and show that $\bh_K(x:g)=\bh$.

First, we observe that by the previous lemma together with (\ref{,ack}),
\begin{eqnarray*}
vg_{\bh}(c) & = & v\sum_{i=1}^m k_i f^{[i]}_{\bhsc}(c) \>=\>
\min\left\{v\sum_{i=1}^m k_id_i\,,\,v\left({\sum_{i=1}^m
(k_i f_h^{[i]}(c)-k_i d_i)}\right)\right\} \\
 & = & v\sum_{i=1}^m k_i d_i \>=\>\min_{1\leq i \leq m} vk_i d_i
\>=\>\min_{1\leq i \leq m} v\,k_i f^{[i]}_{\bhsc}(c)
\end{eqnarray*}
for $c\nearrow x$ (in particular, $vg_{\bh}(c)<\infty$ which implies
that $g$ is nonconstant); with $1\leq j\not=\bh$ we obtain:
\begin{eqnarray*}
v\,g_{\bhsc}(c)(x-c)^{\bhsc}
&=& vg_{\bhsc}(c)+\bh\cdot v(x-c)\>=\>
\left(\min_{1\leq i \leq m} v\,k_i f^{[i]}_{\bhsc}(c)\right)
+\bh\cdot v(x-c)\\
&=& \min_{1\leq i \leq m} v\,k_i f^{[i]}_{\bhsc}(c)(x-c)^{\bhsc}\\
&<& \min_{1\leq i \leq m} v\,k_i f^{[i]}_{j}(c)(x-c)^{j}\\
&\leq& v\sum_{i=1}^m k_i f^{[i]}_{j}(c)(x-c)^j
\>=\> v\,g_{j}(c)(x-c)^{j}\;.
\end{eqnarray*}
This proves that $\bh_K(x:g)=\bh$. It also follows that
\begin{eqnarray*}
\dist(g(x),K) &=& vg_{\bh}(c) + \bh\cdot\dist(x,K)
\>=\>\left(\min_{1\leq i \leq m} vk_i f^{[i]}_{\bh}(c)\right) +
\bh\cdot\dist(x,K)\\
&=& \min_{1\leq i \leq m} v\,k_i f^{[i]}_{\bh}(c)
+ \bh\cdot\dist(x,K)\\
&=& \min_{1\leq i \leq m} v k_i + vf^{[i]}_{\bh}(c)
+ \bh\cdot\dist(x,K)\\
&=& \min_{1\leq i \leq m} v k_i + \dist(f^{[i]}(x),K)
\>=\> \min_{1\leq i \leq m} v k_i + \dist(y_i,K)\\
&\leq& \min_{1\leq i \leq m} v k_i + v(y_i-f^{[i]}(x))
\leq \min_{1\leq i \leq m} v(k_i y_i- k_i f^{[i]}(x))\\
&\leq& v\sum_{i=1}^m (k_i y_i- k_i f^{[i]}(x))
= v\left(\sum_{i=1}^m k_i y_i - g(x)\right)\>,
\end{eqnarray*}
where the first equality follows from Lemma~\ref{LK1b} as $\appr(x,K)$
is transcendental. By Lemma~\ref{v>Lambda} b), this shows that
\[
\dist\left(\sum_{i=1}^m k_i y_i,K\right) \>=\> \dist(g(x),K)
\leq v\left(\sum_{i=1}^m k_i y_i - g(x)\right)\;.
\]
Consequently,
\[
\bh_K(x:\sum_{i=1}^m k_i y_i)=\bh_K(x:g)=\bh\;.
\]
\end{proof}

%
%
\section{Valuation independence of Galois groups}  \label{sectvigp}
In this section, we will introduce a valuation theoretical property that
characterizes the Galois groups of tame Galois extensions. Take a Galois
extension $(L|K,v)$ of henselian fields.
Its Galois group $\Gal L|K$ will be called
\bfind{valuation independent} if for every choice of elements
$d_1,\ldots,d_n\in \tilde{L}$ and automorphisms $\sigma_1, \ldots,
\sigma_n\in \Gal L|K$ there exists an element $d\in L$ such that (for
the unique extension of the valuation $v$ from $L$ to $\tilde{L}$):

\begin{equation}                   \label{,viG}
v\sum_{i=1}^n \sigma_i(d)\,d_i
\>=\> \min_{1\leq i\leq n} v\,\sigma_i(d)\,d_i\;.
\end{equation}

Since $(K,v)$ is assumed to be henselian, we have that $v\sigma (d)=vd$
for all $\sigma\in \Gal L|K$ and therefore, $v\sigma_i(d)\,d_i=vd+vd_i$.
Suppose that $vd_{i_0}=\min_{i} vd_i\,$; then (\ref{,viG}) will hold if
and only if
\[
v\sum_{i=1}^n \frac{\sigma_i(d)}{d}\,\frac{d_i}{d_{i_0}}\>=\>0\;.
\]
In this sum, the terms with $v(d_i/d_{i_0})>0$ have no influence, and we
can delete the corresponding $\sigma_i$ from the list. So we see:

\begin{lemma}
Assume that $(L|K,v)$ is a Galois extension of henselian fields. Then
$\Gal L|K$ is valuation independent if and only if for every choice of
elements $d_i\in \tilde{L}$ with $vd_i=0$ for $1\leq i\leq n$, and
automorphisms $\sigma_1, \ldots, \sigma_n\in \Gal L|K$, there exists an
element $d\in L$ such that
\begin{equation}                            \label{viggh}
v\sum_{i=1}^n \frac{\sigma_i(d)}{d}\,d_i\>=\>0\;.
\end{equation}
\end{lemma}

\begin{theorem}                             \label{+Gtvi}
A Galois extension of henselian fields is tame if
and only if its Galois group is valuation independent.
\end{theorem}
\begin{proof}
Take a Galois extension $(L|K,v)$ of henselian fields, elements $d_i \in
\tilde{L}$ with $vd_i=0$ for $1\leq i\leq n$, and automorphisms
$\sigma_1, \ldots, \sigma_n\in \Gal L|K$.
For $\sigma\in \Gal L|K$ and $d\in L^{\times}$, we set
\[
\chi_\sigma(d)\>:=\>\frac{\sigma(d)}{d}\,v\>.
\]
Since $v\sigma(d)=vd$, the right hand side is a nonzero element in
$Lv$. Now equation (\ref{viggh}) is equivalent to
\begin{equation}                            \label{vichar}
\sum_{i=1}^nd_i v\cdot \chi_{\sigma_i}(d)\>\ne\> 0\,;
\end{equation}
note that $d_i v\ne 0$ since $vd_i=0$.

\pars
We extend the homomorphism
\[
G^i(L|K,v)\ni\sigma\;\mapsto\;\chi_\sigma\in\mbox{\rm Hom}(L^\times,
(Lv)^\times)\;,
\]
which is well known from ramification theory (see \cite{[Eng--P]},
Lemma~5.2.6), to a crossed homomorphism from $\Gal L|K$ to $\mbox{\rm
Hom} (L^\times, (Lv)^\times)$. For the definition and an application
of crossed homomorphisms, see \cite[\S6]{Kadd}. As in the case of
$\sigma\in G^i(L|K,v)$, it is shown that $\chi_\sigma\in \mbox{\rm
Hom}(L^\times, \ovl{L}^\times)$. This group is a right $\Gal L|K$-module
under the scalar multiplication
\[
\chi^\rho\>:=\> \chi\circ\rho\>.
\]
We compute:
\[
\chi_{\sigma\tau}(d)=\frac{\sigma\tau(d)}{d}\,v=
\frac{\sigma\tau(d)}{\tau (d)}\,v\cdot \frac{\tau(d)}{d}\,v=
(\chi_{\sigma}\circ\tau)(d)\cdot \chi_{\tau}(d)\;.
\]
Thus,
\[
\chi_{\sigma\tau}\>=\>\chi_\sigma^\tau\cdot\chi_\tau\;.
\]
In other words, the map
\begin{equation}                            \label{chGZ}
\Gal L|K\ni\sigma\>\mapsto\>\chi_\sigma\in
\mbox{\rm Hom}(L^\times,(Lv)^\times)
\end{equation}
is a crossed homomorphism. Hence, it is injective if and only if its
kernel is trivial. This kernel consists of all $\sigma\in\Gal L|K$
for which $\frac{\sigma(d)}{d}\,v=1$ for all $d\in L^\times$. So
the kernel is the ramification group $G^r(L|K,v)$.

The theorem of Artin\index{Artin} on linear independence of characters
(see \cite{LANG3}, VI, \S4, Theorem~4.1) tells us that if the
$\chi_{\sigma_i}$ are distinct characters, then an element $d$
satisfying (\ref{vichar}) will exist. This shows that $G$ is valuation
independent if the map in (\ref{chGZ}) is injective. The converse is
also true: if $\sigma_1\ne\sigma_2$ but $\chi_{\sigma_1}
=\chi_{\sigma_2}$, then with $n=2$ and $d_1=-d_2=1$, (\ref{viggh}) does
not hold for any $d$.

Since the kernel is the ramification group of $(L|K,v)$, we conclude
that $\Gal L|K$ is valuation independent if and only if the ramification
group is trivial. This is equivalent to $(L|K,v)$ being a tame extension.
\end{proof}

Note that we
could give the above definition and the result of the theorem also for
extensions which are not Galois, replacing automorphisms by embeddings;
however, the normal hull of an algebraic extension $L|K$ of a henselian
field $K$ is a tame extension of $K$ if and only if $L|K$ is a tame
extension, so there is no loss of generality in restricting our scope to
Galois extensions.

%
%
\section{A pull down principle for henselian rationality through tame
extensions}     \label{+8.2}
Take a tame extension $(L|K,v)$ of fields of rank 1 and an immediate
function field $(F|K,v)$ of transcendence degree 1 with $F$ not
contained in the completion $K^c$ of $K$. By Lemma~\ref{idllind}, the
extension $(F^h.L|L,v)$ is again immediate. Since $L|K$ is algebraic,
so is $F^h.L|F^h$ and therefore, $F^h.L$ is henselian, so
$F^h.L=(F.L)^h$. We consider the following question:
\sn
{\it If $F^h.L|L$ is a henselian rational function field,
does this imply the same for $F^h|K$?}

\pars
To start with, we observe that w.l.o.g.\ we may assume the extension
$L|K$ to be finite and Galois. Indeed, if $x\in F^h.L$ such that $F^h.L
=L(x)^h$, then $x$ lies already in $F^h.L_1$ for some finite
subextension $L_1|K$ of $L|K$. Since $x$ must be transcendental over
$L_1\,$, the extension $F^h.L_1|L_1(x)^h$ is finite, generated by
finitely many elements that lie in $L(x)^h$. So we can choose a finite
subextension $L_2|L_1$ of $L|L_1$ such that these elements already lie
in $L_2(x)^h$. Since the normal hull of a tame extension is a tame
extension as well, we may replace $L_2$ by its normal hull $L_3$ over
$K$ because also $L_3(x)^h$ will contain these elements.

From now on we assume that $L|K$ is a finite tame Galois extension and
that $F^h.L=L(x)^h$ for some $x\in F^h.L$. In addition, we assume that
$\appr(x,L)$ is transcendental.

We show that hypothesis (\ref{,sit1}) holds with $K$ replaced by $L$.
First, since $(F.L|L,v)$ is an immediate function field, so is
$(L(x)|L,v)$. Second, $\appr(x,L)$ is transcendental by assumption.
Third, we have:

\begin{lemma}                     
The condition $F\not\subset K^c$ implies that $F.L\not\subset L^c$, hence
$x\notin L^c$.
\end{lemma}
\begin{proof}
Since $F\not\subset K^c$, there exists some $z\in F$ with $z\notin K^c$.
By assumption, $(L|K,v)$ is a tame extension, and as remarked in
Section~\ref{sectprel}, is therefore defectless. Hence by
Lemma~\ref{atup}, $\dist(z,L)=\dist(z,K)<\infty$. Consequently,
$F.L\not\subset L^c$, as asserted.

%

Furthermore, $x\in L^c$ would imply that $L(x)\subset L^c$; since the
rank of $(K,v)$ is 1 by assumption, the same is true for $(L(x),v)$ and
$L(x)$ is thus dense in $L(x)^h$, so we would get that $F.L\subset F^h.L=
L(x)^h\subset L^c$, a contradiction.
\end{proof}

\begin{lemma}                     \label{,yhr}
If there exists an element $y\in F^h$ such that $L(y)^h=L(x)^h$, then
$F^h=K(y)^h$.
\end{lemma}
\begin{proof}
Since $(F^h|K,v)$ and hence also its subextension $(K(y)^h|K,v)$ are
immediate and $(L|K,v)$ is defectless and finite, we obtain from
Lemma~\ref{idllind} that $[F^h.L:F^h]=[L:K]=[K(y)^h.L:K(y)^h]$. On the
other hand, $F^h.L= L(x)^h = L(y)^h = K(y)^h.L$, so $F^h=K(y)^h$ must
hold, because by assumption on $y$, $K(y)^h\subseteq F^h$.
\end{proof}

Since $L|K$ is a finite tame Galois extension, also the extension
$F^h.L|F^h$ is a finite tame Galois extension. As shown in the preceding
proof, it is of degree $n:=[L:K]$. We write
\[
\Gal(F^h.L|F^h)=\{\rho_i\mid 1\leq i\leq n\}\;.
\]
Then $\Gal(L|K)=\{\rho_i\restr L\mid 1\leq i\leq n\}$.

The next lemma will help us to determine the relative approximation
degrees of the conjugates $\rho_i(x)$.

\begin{lemma}                     
Assume that $\rho$ is a valuation preserving automorphism of $L(x)^h$
such that $\rho(L)=L$. Then
\[
L(x)^h = L(\rho x)^h\;.
\]
\end{lemma}
\begin{proof}
Since $\rho x\in \rho(L(x)^h)=L(x)^h$, we have that $L(\rho x)^h
\subseteq L(x)^h$. Further, $L\subseteq \rho^{-1}(L(\rho x)^h) \subseteq
L(x)^h$ and $x\in \rho^{-1}(L(\rho x)^h)$. Thus, $L(x)\subseteq
\rho^{-1} (L(\rho x)^h)$. Since $\rho$ is valuation preserving and
induces an isomorphism from $\rho^{-1}(L(\rho x)^h)$ to the henselian
field $L(\rho x)^h$, also $\rho^{-1}(L(\rho x)^h)$ is henselian; it is
therefore equal to $L(x)^h$. This shows that its image $L(\rho x)^h$
under the automorphism $\rho$ is also equal to $L(x)^h$.
\end{proof}

The following lemma and theorem make essential use of the
valuation independence of Galois groups of tame Galois extensions.
Let Tr denote the trace.
\begin{lemma}                     
There is an element $d\in L$ such that
\[
\bh_K(x:\mbox{\rm Tr}^{}_{F^h.L|F^h}(d\cdot x)) = 1\;.
\]
\end{lemma}
\begin{proof}
From the preceding lemma it follows that every $\rho_i(x)$ is
transcendental over $L$ and hence over $K$. Hence by Lemma~\ref{exapco}
we can choose approximation coefficients $d_i$ of $\rho_i(x)$ in $x$
over $K$ for $1\leq i\leq n$. By Theorem~\ref{+Gtvi}, we have that
$\Gal(L|K)$ is valuation independent. This means we can choose
an element $d\in L$ such that (\ref{,viG}) holds with $\sigma_i=
\rho_i\restr L\,$. Then for $k_i:= \sigma_i(d)=\rho_i(d)$, the
hypothesis (\ref{,ack}) of Lemma~\ref{,acm} holds. In view of the
previous lemma and Corollary~\ref{,11b} we have that
$\bh_K(x:\rho_i(x))=1$. From Lemma~\ref{,acm} we can now infer that
\begin{eqnarray*}
\bh_K\left(x:\mbox{\rm Tr}^{}_{F^h.L|F^h}(d\cdot x)\right) &=&
\bh_K\left(x:\sum_{i} \rho_i(d\cdot x)\right)\\
&=& \bh_K\left(x:\sum_{i} \rho_i(d)\cdot\rho_i(x)\right) = 1\;.
\end{eqnarray*}
\end{proof}

Now we are able to answer our question:
\begin{theorem}                  \label{,pdp}
Let $(K,v)$ be an algebraically maximal field of rank 1, and let $(F,v)$
be an immediate function field of transcendence degree 1 over $(K,v)$,
with $F\not\subset K^c$. If $F^h.L$ is a henselian rational function
field over $L$ for some tame extension $(L|K,v)$, then $F^h$ is a
henselian rational function field over $K$.
\end{theorem}
\begin{proof}
As shown in the beginning of this section, we may assume that $L|K$ is
finite and Galois. Now the foregoing lemma shows that there is some
$d\in L$ such that for $y:=\mbox{\rm Tr}^{}_{F^h.L|F^h} (d\cdot x)\in
F^h$ we have $\bh_K(x:y)=1$. By virtue of Corollary~\ref{,11b},
$L(y)^h=L(x)^h$. From Lemma~\ref{,yhr}, we can now infer that $F^h$ is
henselian rational over $K$, as asserted.
\end{proof}


\begin{thebibliography}{99}

\bibitem{End} Endler, O.: {\it Valuation theory}, Springer, Berlin
1972

\bibitem{[Eng--P]} Engler, A.J.\ -- Prestel, A.: {\it Valued fields},
Springer Monographs in Mathematics. Springer-Verlag, Berlin, 2005

\bibitem{KAP1} Kaplansky, I.: {\it Maximal fields with valuations I},
Duke Math.\ Journ.\ {\bf 9} (1942), 303--321


\bibitem{KK2} Knaf, H. - Kuhlmann, F.-V.: {\it Every place admits
local uniformization in a finite extension of the function field},
Advances in Math. 221 (2009), 428-453

\bibitem{FVKth} Kuhlmann, F.--V.: {\it Henselian function fields
and tame fields}, extended version of Ph.D.\ thesis, Heidelberg
1990

\bibitem{Kadd} Kuhlmann, F.-V.: {\it Additive Polynomials and Their Role
in the Model Theory of Valued Fields}, Proceedings of the Workshop and
Conference on Logic, Algebra, and Arithmetic, 2003.
Lecture Notes in Logic 26 (2006), 160--203

\bibitem{Kasd} Kuhlmann, F.-V.: {\it A classification of Artin Schreier
defect extensions and characterizations of defectless fields},
Illinois J.\ Math.\ {\bf 54} (2010), 397--448

\bibitem{Kdef} Kuhlmann, F.-V.: {\it The defect}, in: Commutative
Algebra - Noetherian and non-Noetherian perspectives. Marco Fontana,
Salah-Eddine Kabbaj, Bruce Olberding and Irena Swanson (eds.), Springer
2011

\bibitem{FVK} Kuhlmann, F.-V.: {\it The algebra and model theory of
tame valued fields}, submitted

\bibitem{HR} Kuhlmann, F.-V.: {\it Elimination of Ramification II:
Henselian Rationality of Valued Function Fields}, in preparation

\bibitem{LANG3} Lang, S.: {\it Algebra}, Graduate Texts in Mathematics.
Springer-Verlag, New York, 2002

\bibitem{OS3} Ostrowski, A.: {\it Untersuchungen zur arithmetischen
Theorie der K\"orper}, Math.\ Z.\ {\bf 39} (1935), 269--404

\bibitem{PZ} Prestel, A.\ -- Ziegler, M.: {\it Model theoretic
methods in the theory of topological fields}, J.\ reine angew.\ Math.\
{\bf 299/300} (1978), 318--341

\bibitem{RIB1} Ribenboim, P.: {\it Th\'{e}orie des valuations}, Les
Presses de l'Universit\'{e} de Montr\'{e}al, Montr\'{e}al (1968)

\bibitem{Tem} Temkin, M.: Inseparable local uniformization.
preprint, arXiv:[0804.1554]

\bibitem{W} Warner, S.: {\it Topological fields}, Mathematics
studies {\bf 157}, North Holland, Amsterdam (1989)

\bibitem{ZS} Zariski, O.\ -- Samuel, P.: {\it Commutative
Algebra}, Vol.\ II, New York--Heidel\-berg--Berlin (1960)

\end{thebibliography}
\end{document}